\documentclass[12pt]{amsart}
\usepackage{geometry} 
\usepackage{epsfig}
\usepackage{amsmath}
\usepackage{amssymb}
\usepackage{graphicx}
\usepackage{color}
\usepackage{cancel}

\newtheorem{theorem}{Theorem}[section]

\newtheorem{lemma}[theorem]{Lemma}

\newtheorem{remark}[theorem]{Remark}

\newcommand{\eps}{\varepsilon}

\renewcommand{\tilde}{\widetilde}

\newcommand{\Z}{\mathbb Z}
\newcommand{\R}{\mathbb R}

\newcommand{\p}{\partial}

\usepackage{verbatim}

\title{The KdV, the Burgers,  and the Whitham limit  for \\ a spatially periodic Boussinesq model}
\author{Roman Bauer, Wolf-Patrick D\"ull, Guido Schneider}
\date{\today{}} 

\begin{document}

\maketitle

\begin{center}
{\footnotesize IADM, Universit\"at Stuttgart, Pfaffenwaldring 57, 70569 Stuttgart Germany, \\ email: forename.lastname@mathematik.uni-stuttgart.de}
\end{center}

\tableofcontents
%

\begin{abstract}
We are interested in  the Korteweg-de Vries (KdV), the Burgers,  and  the Whit\-ham  limit 
for a spatially periodic Boussinesq model with non-small contrast.
We prove estimates between the KdV, the Burgers,  and  the Whitham approximation and true solutions of the original system
which guarantee that  these amplitude equations  make correct predictions about the dynamics of the spatially periodic Boussinesq model 
 over the natural time scales of the amplitude equations. 
 The proof is based on Bloch wave analysis and energy estimates. 
 The result is the first  justification result of  the KdV, the Burgers,  and the Whitham approximation 
for a dispersive PDE  posed in a spatially periodic medium of non-small contrast.
\end{abstract}

\section{Introduction}
In the long wave limit 
there exists a zoo of amplitude equations which can be derived via multiple scaling analysis 
for various dispersive wave systems with conserved quantities. Generically,
among these amplitude equations there are only three nonlinear ones which are 
independent of the small perturbation parameter, namely the Korteweg-de Vries (KdV) equation, 
the inviscid Burgers equation, 
 and the 
Whitham system. 
It is the purpose of this paper to discuss the validity of these approximations 
for a spatially periodic Boussinesq model with non-small contrast.

\subsection{The formal approximations in the spatially homogenous situation}

The KdV equation occurs as an amplitude equation in the description of small spatially  and  temporally modulations
of long waves in various dispersive wave systems. Examples are the water wave problem or equations 
from plasma physics, cf. \cite{CS98}. 
For the  Boussinesq equation 
\begin{equation}
\label{const}
\partial_t^2 u(x,t)
=
\partial_x^2 u(x,t)
-\partial_x^4 u(x,t)
+\partial_x^2(u(x,t)^2),
\end{equation}
with $ x \in \mathbb{R} $, $ t \in \mathbb{R} $, and $ u(x,t) \in \mathbb{R} $, by the ansatz
\begin{equation}
\label{ansatz}
u(x,t) = \varepsilon^2 A(X,T) , 
\end{equation}
where $ X = \varepsilon(x-t) $,  $T = \varepsilon^3 t $, $ A(X,T) \in \mathbb{R} $, and $ 0 < \varepsilon \ll 1 $ a small perturbation parameter, the KdV equation
\begin{equation}
\label{constkdv1}
\partial_T A = 
\frac{1}{2}\partial_X^3 A
-\frac{1}{2}\partial_X(A^2)
\end{equation}
can be derived by inserting \eqref{ansatz} into \eqref{const} and by equating the coefficients 
in front of $ \varepsilon^6 $ to zero. 
This ansatz can be generalized to 
\begin{equation}
u(x,t) = \varepsilon^{\alpha} A(X,T) , 
\end{equation}
where $ X = \varepsilon(x-t) $,  $T = \varepsilon^{1+\alpha} t $, and $ A(X,T) \in \mathbb{R} $, 
with $ \alpha > 0 $. For $ \alpha > 2 $ the Airy equation 
$ \partial_T A = \frac{1}{2}\partial_X^3 A $ occurs. The KdV equation is recovered for $ \alpha = 2 $,
and for $ \alpha \in (0,2) $   the inviscid Burgers equation 
\begin{equation}
\label{constburgers1}
\partial_T A = 
-\frac{1}{2}\partial_X(A^2)
\end{equation}
is obtained.
There is another long wave limit which leads to an $ \varepsilon $-independent  non-trivial amplitude equation.
With the ansatz 
\begin{equation}
\label{ansatzwhit}
u(x,t) =  U(X,T) , 
\end{equation}
where $ X = \varepsilon x $,  $T = \varepsilon t $, and $ U(X,T) \in \mathbb{R} $, we obtain
\begin{equation}
\label{constwhit1}
\partial_T^2 U = 
\partial_X^2 U
+\partial_X^2(U^2)
\end{equation}
which can be written as a system of conservation laws
\begin{equation}
\label{constwhit2}
\partial_T U = \partial_X V , \qquad  \partial_T V =
\partial_X U
+\partial_X(U^2).
\end{equation}
In the following both, \eqref{constwhit1} and \eqref{constwhit2}, are  called the Whitham system, cf. \cite{Wh74}.

\subsection{Justification by error estimates}

Estimates that the formal KdV approximation and true solutions of the original system stay close together over the natural KdV time scale
are a non-trivial task  since solutions of order $ \mathcal{O}(\varepsilon^2) $ have to be shown to be existent on an  $ \mathcal{O}(1/\varepsilon^3) $ 
time scale.
For \eqref{const} an approximation result is formulated as follows.
\begin{theorem} \label{th11}
Let $s\geq 0$ and let $A \in C([0,T_0],H^{5+s})$ be a solution of the KdV equation \eqref{constkdv1}. Then there exist $ \varepsilon_0 > 0 $ and $ C > 0 $ such that
for all $ \varepsilon \in (0,\varepsilon_0)$ we have solutions $ u $ of \eqref{const} with 
$$ 
\sup_{t \in [0,T_0/\varepsilon^3]} \| u(\cdot,t) - \varepsilon^2 A(\varepsilon(\cdot-t), \varepsilon^3 t ) \|_{H^{1+s}} \leq C \varepsilon^{7/2}. 
$$
\end{theorem}
There are two fundamentally different approaches 
to prove such an approximation result. 
For analytic initial conditions of the KdV equation  a Cauchy-Kowalevskaya 
based approach can be chosen, see \cite{KN86} with the comments given in \cite{Schn95ICIAM} for the water wave problem. 
Working in spaces of analytic functions gives some artificial smoothing which allows to gain the 
missing order w.r.t. $ \varepsilon $ between 
the inverse of the amplitude of $ \mathcal{O}(\varepsilon^2)$ and the time scale 
of $ \mathcal{O}(1/\varepsilon^3)$ via the derivative 
in front of the nonlinear terms in the KdV equation.
This approach is very robust 
and works without a detailed analysis of the underlying problem, cf. \cite{CDS14} for another example, but gives not optimal results. 

For initial conditions in Sobolev spaces the 
underying idea to gain such estimates is 
conceptually rather simple,
namely the construction of a suitable chosen
energy which include the terms 
of order  $ \mathcal{O}(\varepsilon^2)$ in the 
equation for the error, such that for the energy 
finally $ \mathcal{O}(\varepsilon^3 t)$ growth rates 
occur. However, the method is less robust 
since for every single original system a
different energy occurs and the major difficulty is 
the construction of this energy.
Estimates that the formal KdV approximation and true solutions of the different formulations 
of the water wave problem stay close together over the natural KdV time scale have been shown for instance in \cite{Cr85,SW00,SW02,BCL05,Du12} using this approach.
Another example is the justification of the KdV 
approximation for modulations of periodic waves 
in the NLS equation, cf. \cite{CR10}.
For \eqref{const} the energy approach is rather short and very instructive for the subsequent analysis.
Therefore, we recall it in Section \ref{sec2}.

Interestingly, it turns out that the proofs given for the KdV approximations transfer more or less line for 
line into proofs for the justification of the 
inviscid Burgers equation and of the Whitham system.
Since only the scaling has to be adapted, whenever 
a KdV approximation result holds also an 
inviscid Burgers  and Whitham 
approximation result can be established.
This will be explained in detail in Section \ref{sec2}.

As above such approximation results are a non-trivial task  since solutions of order $ \mathcal{O}(\varepsilon^{\alpha}) $ have to be shown to be existent on an  $ \mathcal{O}(\varepsilon^{1+\alpha}) $ 
time scale. For the inviscid Burgers equation the 
formulation of the approximation result 
goes along the lines of Theorem \ref{th11}. However, due to the notational complexity in achieving in general 
the estimates  for the residual
(the terms which do not cancel
after inserting the approximation into   \eqref{Boussinesq}),
 in Remark \ref{alphaville} we restrict ourselves to the case $ \alpha = 1 $.
\begin{theorem} \label{th12}
Let $s\geq 0$, $\alpha = 1 $ and let $A \in C([0,T_0],H^{3+s})$ be a solution of the inviscid Burgers  equation \eqref{constburgers1}. Then there exist $ \varepsilon_0 > 0 $ and $ C > 0 $ such that
for all $ \varepsilon \in (0,\varepsilon_0) $ we have solutions $ u $ of \eqref{const} with 
$$ 
\sup_{t \in [0,T_0/\varepsilon^2]} \| u(\cdot,t) - \varepsilon A(\varepsilon(\cdot-t), \varepsilon^2 t ) \|_{H^{1+s}} \leq C \varepsilon^{(3+2 \alpha)/2}. 
$$
\end{theorem}

Since for the Whitham approximation  solutions of order 
$ \mathcal{O}(1)$ are considered some smallness
condition is needed such that the used energy 
allows us to estimate the associated Sobolev norm. 

For \eqref{const} a possible Whitham approximation result is formulated as follows.
\begin{theorem} \label{th11whit}
Let $s\geq 0$. There exists a $ C_1 > 0 $ such that the 
following holds.  Let $U \in C([0,T_0],H^{3+s})$ be a solution of   \eqref{constwhit1}
with 
$$ \sup_{T \in [0,T_0]} \| U(\cdot,T) \|_{H^{3+s}} \leq C_1 .$$
 Then there exist $ \varepsilon_0 > 0 $ and $ C > 0 $ such that
for all $ \varepsilon \in (0,\varepsilon_0)$ we have solutions $ u $ of \eqref{const} with 
$$ 
\sup_{t \in [0,T_0/\varepsilon]} \| u(\cdot,t) - U(\varepsilon \cdot, \varepsilon t ) \|_{H^{1+s}} \leq C \varepsilon^{3/2}. 
$$
\end{theorem}
The Whitham system for the water wave problem 
coincides with the shallow water wave equations 
which have been justified for the water wave problem 
without surface tension
in \cite{Ov76,Ig07}.
A Whitham approximation result 
that the periodic wave trains of the NLS equation 
are approximated by the Whitham system can be 
found in \cite{DS09}.

\subsection{The spatially periodic situation}

The last years have seen some first attempts to justify the KdV equation in periodic media. 
It has been justified in \cite{Ig07} for the water wave problem over a periodic bottom in the KdV scaling, i.e., with long wave  oscillations of 
the bottom of magnitude $ \mathcal{O}(\varepsilon^2) $ varying on a spatial scale of order
$ \mathcal{O}(\varepsilon^{-1}) $.
 The same result can be found in \cite{Ch09}
where general bottom topographies of small amplitude have  been handled.
The result is based on  \cite{Ch07}
where other  amplitude systems have been justified.
This situation can be handled as perturbation 
of the spatially  homogeneous  case.

In case of oscillations of the bottom of magnitude  $ \mathcal{O}(1) $ varying on a spatial scale of  order $ \mathcal{O}(1) $,  no approximation result 
can be found in the existing literature.
As a first attempt to solve this question for the water wave problem we consider 
a spatially periodic Boussinesq equation
\begin{align}
\label{Boussinesq}
 \partial_t^2 u(x,t) 
= &
\partial_x(a(x)\partial_x u(x,t)) \\ &
-\partial_x^2(b(x)\partial_x^2u(x,t) )
+\partial_x(c(x)\partial_x(u(x,t)^2)), \nonumber
\end{align}
with $ x \in \mathbb{R} $, $ t \geq 0 $, $ u(x,t) \in \mathbb{R} $, and smooth $ x $-dependent
$ 2 \pi $-spatially periodic coefficients $ a $, $ b $, and $ c $ satisfying 
$$
\inf_{x \in \mathbb{R}} a(x) > 0 \quad \textrm{and} \quad 
 \inf_{x \in \mathbb{R}} b(x)>0. 
 $$  
For this equation we derive the KdV equation 
by making a Bloch mode expansion of 
\eqref{Boussinesq}. The KdV approximation describes 
the modes which in Figure \ref{fig:example} are contained 
in the  circles.
We prove an approximation result which is formulated 
in Theorem \ref{main_theorem}.
It guarantees that  the KdV equation  makes correct predictions about the dynamics of the spatially periodic Boussinesq model \eqref{Boussinesq}
 over the natural KdV time scale. 
The presented result is the first  justification result of  the KdV approximation 
for a dispersive nonlinear PDE  posed in a spatially periodic medium of non-small contrast.
For linear systems this limit has been considered independently in \cite{DLS14,DLS15}.

In order  to make the residual small an improved approximation 
has to be constructed. 
Since this construction is not the main purpose of this paper
we additionally assume 

{\bf (SYM)}
 the coefficient functions 
$$
a= a(x), \quad b= b(x) , \quad \textrm{and} \quad c= c(x) \quad \textrm{are even w.r.t. }  x.
$$
As in the spatially homogeneous situation 
it turns out that the proof given for the KdV approximation transfers more or less line for 
line into proofs for the justification of the approximation   via the 
inviscid Burgers equation and of the Whitham system.
The associated approximation results are  formulated in Theorem \ref{main_theoremburgers}
and Theorem \ref{main_theoremwhitham}.

The paper was originally  intended as the next step in generalizing a method 
which has been developed in \cite{CS11} for the justification of the KdV approximation 
in situations when the KdV modes are resonant to other long wave modes.
The method had already successfully been 
applied in justifying the KdV approximation for the poly-atomic FPU problem in \cite{CCPS12}. 
The qualitative difference in justifying the KdV equation for the spatially periodic Boussinesq equation in contrast to \cite{CS11,CCPS12} is that 
for fixed Bloch respectively  Fourier wave number 
the presented problem is infinite dimensional. \cite{CS11,CCPS12}  corresponds to the middle panel of Figure \ref{fig:example} where the spatially periodic Boussinesq equation
corresponds to the right panel of Figure \ref{fig:example}.
As a consequence the normal form transform which is a major part 
of the proofs of \cite{CS11,CCPS12}  would be more demanding from an analytic point of view. In the justification of the Whitham system  with the approach of \cite{CS11,CCPS12}
infinitely many normal form transform\textcolor{red}{s} have to be performed \cite{DSS16}.

   \begin{figure}[htbp] 
   \centering
   \includegraphics[width=1.5in]{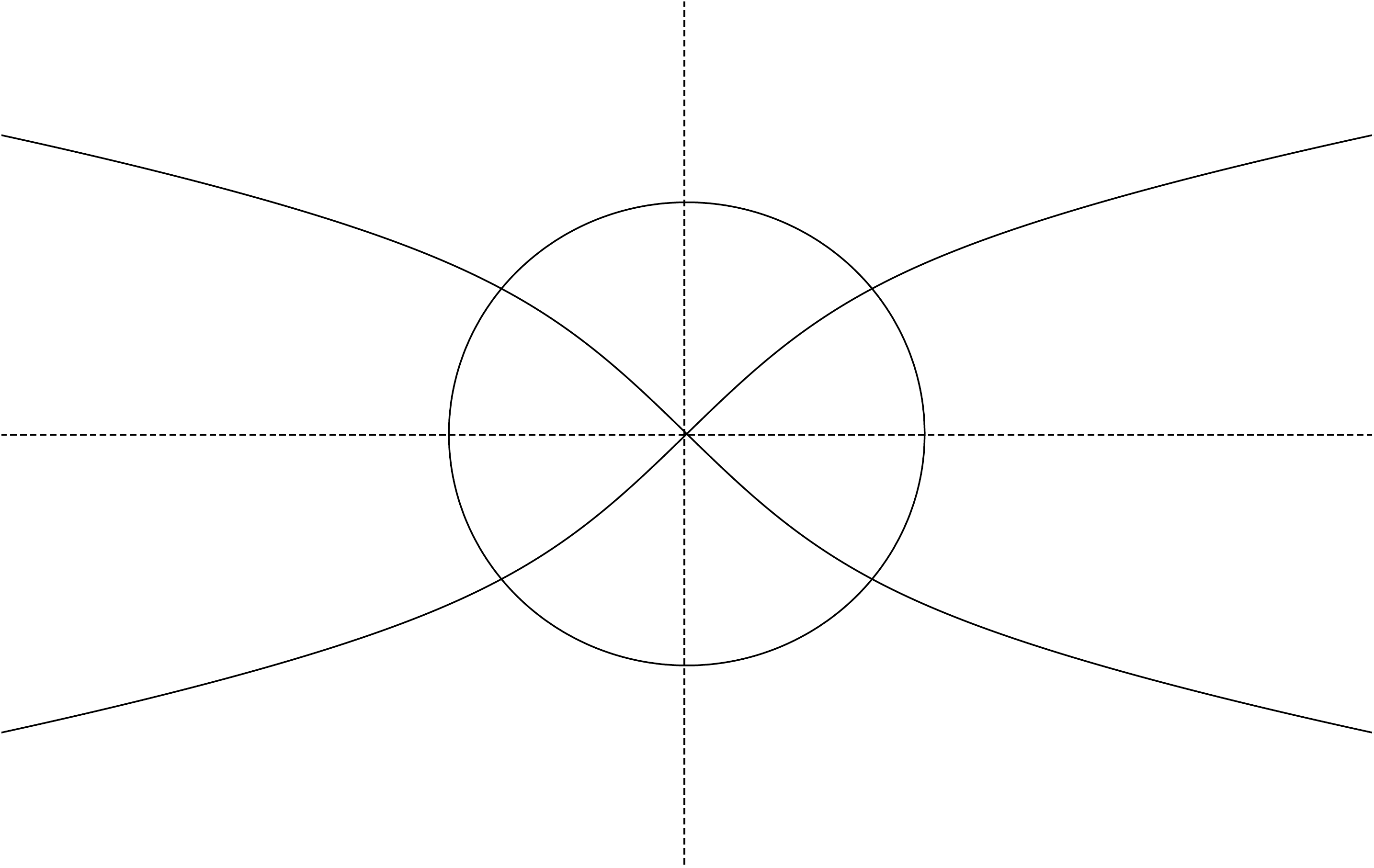} \qquad
   \includegraphics[width=1.5in]{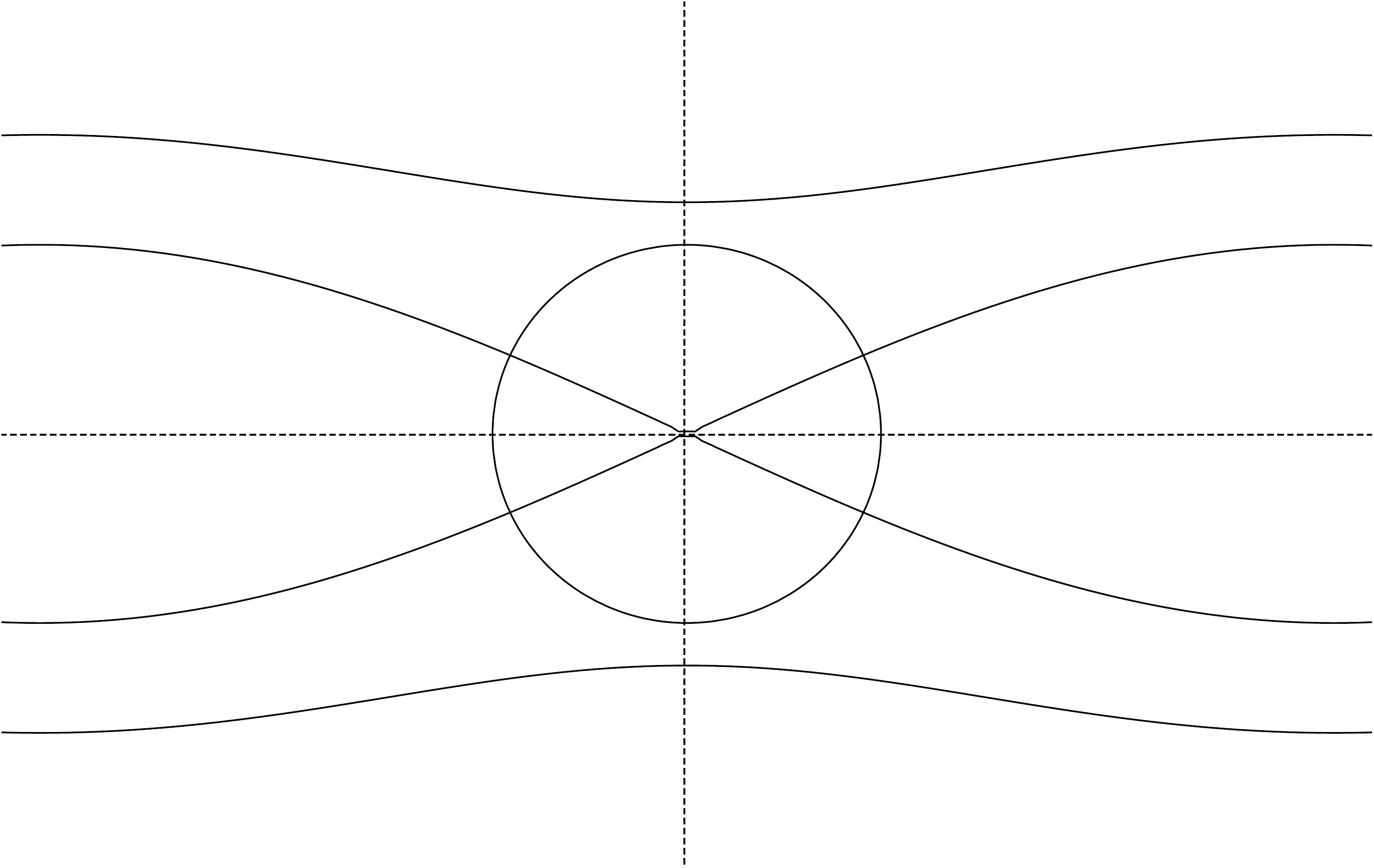} 
   \qquad
   \includegraphics[width=1.5in]{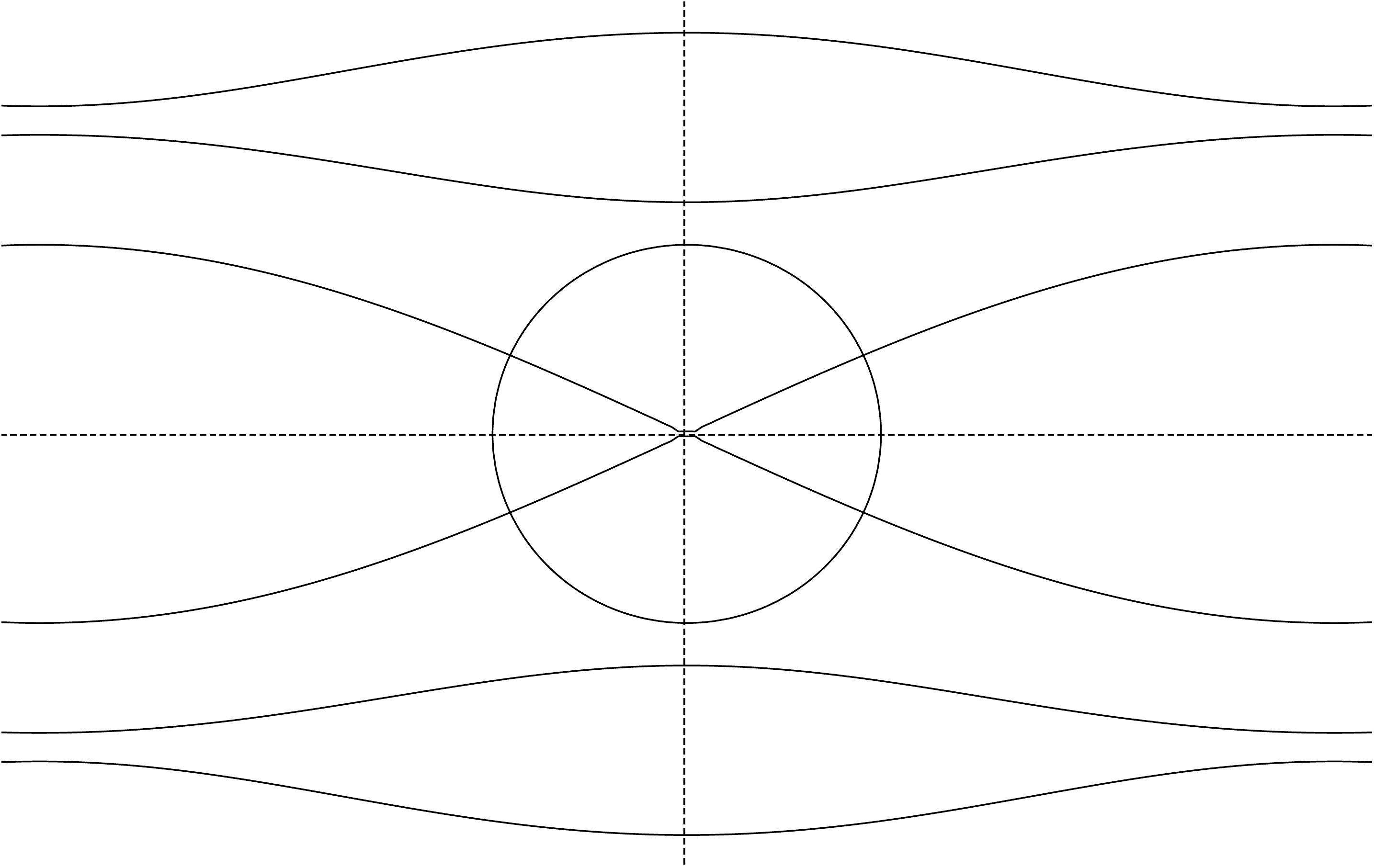} 
   \caption{The left panel shows the curves of eigenvalues over the Fourier wave numbers 
  as it appears for the water wave problem \cite{Cr85,SW00,SW02,BCL05,Ig07,Du12}.
  The middle panel shows the  finitely many curves of eigenvalues as they  appear for instance for the poly-atomic 
   FPU system \cite{CS11,CCPS12}. The right panel shows the  infinitely many curves of eigenvalues 
  over the Bloch wave numbers 
   as it appears for the  spatially periodic Boussinesq model \eqref{Boussinesq},
   the water wave problem over a periodic bottom topography, or for the linearization 
   around a periodic wave in dispersive systems.
   Since the Fourier transform of $ \varepsilon^2 A(\varepsilon x) $ is given by 
   $ \varepsilon^2 \varepsilon^{-1} \widehat{A}( x/\varepsilon) $
   the KdV equations describe the modes at the wave numbers $ k = 0 $ 
   with the vanishing eigenvalues which are contained in the circles. One of the two curves in the circle describes wave packets moving to the left, the other curve
   wave packets moving to the right.}
   \label{fig:example}
\end{figure}

Interestingly, for the spatially periodic Boussinesq equation \eqref{Boussinesq} there exists 
 an energy in physical space
which allowed us
to incorporate the normal form transforms into the energy estimates.
This energy approach is presented in the following.
\medskip

{\bf Notation.} Constants which can be chosen independently of the small perturbation parameter $ 0 < \varepsilon \ll 1 $
are denoted with the same symbol $ C $. 
We write $ \int $ for $ \int_{-\infty}^{\infty} $. The Fourier transform of a function $ u $ is denoted with $ \widehat{u}  $.
The Bloch  transform of a function $ u $ is denoted with $ \widetilde{u}  $ and this tool is recalled in Appendix \ref{leutkirch}.
We introduce the norm $ \| \cdot \|_{L^2_s} $ by
$$
\| \widehat{u} \|_{L^2_s}^2 = \int | \widehat{u}(k) |^2 (1+k^2)^s dk
$$
and define the Sobolev norm $ \| u \|_{H^s} = \| \widehat{u} \|_{L^2_s} $, but use also equivalent versions.
\medskip

{\bf Acknowledgement.} 
The authors are grateful to Florent Chazel for helping us to understand the existing literature.
Moreover, we would like to thank Martina Chirilus-Bruckner for a number of helpful discussions.
The paper is partially supported by the Deutsche Forschungsgemeinschaft (DFG)
under the grant Schn520/9-1.

\section{The spatially homogeneous case}

\label{sec2}

It is the goal of this section to give a simple proof for  Theorem \ref{th11},   Theorem \ref{th12}, and Theorem \ref{th11whit} using the energy method.
The proof will be the basis of the subsequent analysis.
All three cases can be handled with the same approach.

The residual 
$$ 
\mathrm{Res}(u) = -\partial_t^2 u(x,t)
+
\partial_x^2 u(x,t)
-\partial_x^4 u(x,t)
+\partial_x^2(u(x,t)^2)
$$ 
quantifies how much a function $ u $ fails to satisfy the  Boussinesq model \eqref{const}.
For the KdV approximation \eqref{ansatz} abbreviated with $ \varepsilon^2 \Psi $ 
we find  
\begin{eqnarray*}
\mathrm{Res}(\varepsilon^2 \Psi) &  = & - \varepsilon^4 c^2 \partial_X^2 A - 2 \varepsilon^6 \partial_T \partial_X A 
- \varepsilon^8 \partial_T^2 A  \\ && + \varepsilon^4 \partial_X^2 A - \varepsilon^6 \partial_X^4 A 
+ \varepsilon^6 \partial_X^2 (A^2) \\ & = & - \varepsilon^8 \partial_T^2 A
\end{eqnarray*}
if we choose $ A $ to satisfy the  KdV equation  \eqref{constkdv1}.
Therefore, we have
\begin{lemma}
Let $s\geq 0$ and let  $A  \in C([0,T_0],H^{5+s})$ 
be  a solution  of the  KdV equation \eqref{constkdv1}. 
 Then there exist $ \varepsilon_0 > 0 $, $ C_{\rm res} $ such that for all $  \varepsilon \in (0,\varepsilon_0) $ we have 
$$ 
\sup_{t \in [0,T_0/\varepsilon^3]} \|  \partial_x^{s-1}  \mathrm{Res}(   \varepsilon^2 \Psi(\cdot,t,\varepsilon)    )    \|_{L^2} 
\leq C_{\rm res}  \varepsilon^{(13+2s)/2}.
$$
\end{lemma}
\begin{proof}
Using the KdV equation allows us to write 
\begin{eqnarray*}
4 \partial_T^2 A  & = & - 2 \partial_T ( \partial_X^3 A + \partial_X (A^2)) 
= - 2 ( \partial_X^3 \partial_T A + 2 \partial_X (A \partial_T A))\\
&=&\partial_X^3 (\partial_X^3 A + \partial_X (A^2)) +2 \partial_X (A (\partial_X^3 A + \partial_X (A^2))).
\end{eqnarray*}
This shows that 
$A(\cdot,T) \in H^{6}$ is necessary to estimate the residual in $ L^2 $. The formal error  of order $ \mathcal{O}(\varepsilon^8) $ is reduced by a factor $  \varepsilon^{-1/2} $ due to the scaling properties 
of the $ L^2 $-norm. Moreover, 
due to the representation of $ \partial_T^2 A $ as a spatial derivative, below, 
we can  apply $ \partial_x^{-1}  = \varepsilon^{-1}\partial_X^{-1} $ to the residual terms which however loses another factor $ \varepsilon^{-1} $.
\end{proof}

Similarly,  for the Whitham approximation  \eqref{ansatzwhit} abbreviated with $ \varepsilon^2 \Psi $  we find 
$
\textrm{Res}( \Psi) =  - \varepsilon^4 \partial_X^4 U
$
if we choose  $ U $ to satisfy the Whitham system \eqref{constwhit1}.
Hence, for an estimate in $ L^2 $
we need $U\in H^{4}$.  Exactly as above we have 
\begin{lemma}
Let $s\geq 0$ and let  $A  \in C([0,T_0],H^{3+s})$     
be  a solution  of the  Whitham system \eqref{constwhit1}. 
 Then there exist $ \varepsilon_0 > 0 $, $ C_{\rm res} $ such that for all $  \varepsilon \in (0,\varepsilon_0) $ we have 
$$ 
\sup_{t \in [0,T_0/\varepsilon]} \|  \partial_x^{s-1}  \mathrm{Res}(   \Psi(\cdot,t,\varepsilon)    )    \|_{L^2} 
\leq C_{\rm res}  \varepsilon^{(5+2s)/2}.
$$
\end{lemma}
\begin{remark} \label{alphaville} {\rm
For the inviscid Burgers equation the residual becomes too large with the simple ansatz \eqref{ansatz}. However,
 by adding 
higher order terms to the approximation \eqref{ansatz},  with a slight abuse of notation this approximation is again called $ \varepsilon^{\alpha} \Psi $,  one can always achieve 
$$ 
\sup_{t \in [0,T_0/\varepsilon^{1+\alpha}]} \|    \mathrm{Res}(  \varepsilon^{\alpha}   \Psi(\cdot,t,\varepsilon)    )    \|_{L^2} 
\leq C_{\rm res}  \varepsilon^{(7+4\alpha)/2}
$$
and 
$$ 
\sup_{t \in [0,T_0/\varepsilon^{1+\alpha}]} \|  \partial_x^{-1}  \mathrm{Res}( \varepsilon^{\alpha}    \Psi(\cdot,t,\varepsilon)    )    \|_{L^2} 
\leq C_{\rm res}  \varepsilon^{(5+4\alpha)/2}.
$$
See Appendix \ref{appA} where we prove these estimates for $ \alpha = 1 $ and explain that 
the number of additional terms goes to infinity for $ \alpha \to 0 $ and $ \alpha \to 2 $.
}\end{remark}
From this point on the remaining estimates can be handled exactly the same. 
The case $ \alpha = 0 $ corresponds to the Whitham approximation and the case $ \alpha = 2 $ to the KdV approximation. 
The difference $ \varepsilon^{(3+2 \alpha)/2} R = u - \varepsilon^{\alpha} \Psi $
 satisfies 
\begin{equation} \label{smitseq}
\partial_t^2 R
=
\partial_x^2 R
-\partial_x^4 R
+ 2 \varepsilon^{\alpha} \partial_x^2(\Psi R) + \varepsilon^{(3+2 \alpha)/2}\partial_x^2(R^2) + \varepsilon^{-(3+2 \alpha)/2} \mathrm{Res}(   \varepsilon^2 \Psi).
\end{equation}
We multiply the error equation \eqref{smitseq} with $ -\partial_t \partial_x^{-2} R $ which is defined via its Fourier transform w.r.t. $ x$, namely via $ \widehat{\partial_x^{-1} R}(k) = \frac{1}{ik} \widehat{R}(k) $, 
integrate it w.r.t. $ x $, and find 
\begin{eqnarray*}
-\int (\partial_t  \partial_x^{-2} R) \partial_t^2 R dx & = &  \partial_t  \int 
(\partial_t  \partial_x^{-1} R)^2 dx/2, \\ 
-\int (\partial_t  \partial_x^{-2} R) \partial_x^2 R dx & = & - \partial_t  \int 
R^2 dx/2, \\ 
\int (\partial_t  \partial_x^{-2} R) \partial_x^4 R dx & = & - \partial_t  \int 
(\partial_x R)^2 dx/2, \\ 
-\int (\partial_t  \partial_x^{-2}  R) \partial_x^2(\Psi R) dx & = & 
-\int (\partial_t  R) \Psi  R dx \\&=&
- \partial_t  \int  \Psi R^2 dx/2 
+\varepsilon \int  (\partial_{\tau} \Psi)  R^2 dx,\\
-\int (\partial_t  \partial_x^{-2} R) \partial_x^2(R^2) dx & = & - 
\int (\partial_t  R) R^2  dx = - \frac{1}{3}\partial_t   \int  R^3  dx,\\
-\int (\partial_t  \partial_x^{-2}  R) \mathrm{Res}(   \varepsilon^2 \Psi) dx
& = &  \int (\partial_t  \partial_x^{-1}  R) \partial_x^{-1} \mathrm{Res}(   \varepsilon^2 \Psi) dx.
\end{eqnarray*}
We can estimate
\begin{eqnarray*}
|\int (\partial_t  \partial_x^{-1}  R)  \partial_x^{-1}  \mathrm{Res}(   \varepsilon^2 \Psi) dx | & \leq &
\| \partial_t  \partial_x^{-1}  R \|_{L^2} \| \partial_x^{-1}  \mathrm{Res}(   \varepsilon^2 \Psi) \|_{L^2},\\
|\int(\partial_{\tau} \Psi) R^2 dx| & \leq &\| \partial_{\tau} \Psi \|_{L^{\infty}}\|  R \|_{L^2}^2 .
\end{eqnarray*}
For the energy 
\begin{eqnarray*} 
E&= & \int  (\partial_t \partial_x^{-1} R)^2 +R^2 +  (\partial_x R)^2  +2 \varepsilon^{\alpha} \Psi   R^2 + 2 \varepsilon^{(3+2 \alpha)/2} R^3/3 
dx \nonumber
\end{eqnarray*}
the following holds.
In case $ \alpha > 0 $ we have that for all $ M > 0 $ there exist $ C_1,\varepsilon_1 > 0 $ such that for all
$  \varepsilon \in (0,\varepsilon_1 ) $ we have 
$$ 
\|  R \|_{H^1} \leq C_1 E^{1/2}
$$ 
as long as $ E \leq M $. In case $ \alpha = 0 $  the energy $ E $ is an upper bound for the squared $ H^1 $-norm for $ \|\Psi \|_{L^{\infty}}$ sufficiently small, but independent of $ 0 < \varepsilon \ll 1 $.
Therefore, $ E $ satisfies the inequality 
\begin{eqnarray} \label{smitseq4}
\frac{dE}{dt} & \leq &  C \varepsilon^{1+\alpha} E + C \varepsilon^{(3+ 2\alpha)/2}      E^{3/2} + C  \varepsilon^{1+\alpha} E^{1/2}  
\\�& \leq  & 2 C \varepsilon^{1+\alpha} E + C \varepsilon^{(3+ 2\alpha)/2}      E^{3/2} + C  \varepsilon^{1+\alpha}, \nonumber
\end{eqnarray}
with a constant $ C $ independent of $   \varepsilon \in (0,\varepsilon_1 ) $.
Under the assumption that $ C \varepsilon^{1/2}      E^{1/2}  \leq 1 $ we obtain 
$$ 
\frac{dE}{dt} \leq (2 C+1) \varepsilon^{1+\alpha} E + C  \varepsilon^{1+\alpha}.
$$
Gronwall's inequality immediately gives the bound 
$$ 
\sup_{t \in [0,T_0/\varepsilon^{1+\alpha}]} E(t) = C T_0 e^{ (2 C+1)T_0} =: M = \mathcal{O}(1).
$$ 
Finally choosing $ \varepsilon_2 > 0 $ so small that $ C \varepsilon_2^{1/2}      M^{1/2}  \leq 1 $ gives the required estimate for all $  \varepsilon \in (0,\varepsilon_0 ) $
with  $ \varepsilon_0 = \min ( \varepsilon_1 , \varepsilon_2 )> 0 $ in all three cases.

\begin{remark}{\rm
The Boussinesq model \eqref{const} is a semilinear  dispersive system 
and so there is the local existence and uniqueness of solutions.
The variation of constant formula associated to the first order system for the variables 
$ u $ and $ \partial_t (\partial_x^4-\partial_x^2)^{-1/2} u $ is a contraction in the space 
$ C([-T_*,T_*],H^{\theta} \times H^{\theta}) $ for every $ \theta > 1/2 $ if $ T_* > 0 $ is sufficiently small. 
The local  existence and uniqueness of solutions combined with the previous estimates 
for instance 
yields the existence and uniqueness of solutions for all $ t \in [0,T_0/\varepsilon^3] $ in the 
KdV case and all $ t \in [0,T_0/\varepsilon] $ in the Whitham case. 
}
\end{remark}

\section{Derivation of the amplitude equations}

In this section we come back to the spatially periodic situation.
The derivation of the amplitude equations is less obvious than in the spatially homogeneous 
case. In order to derive the amplitude equations we expand \eqref{Boussinesq} into the  
eigenfunctions of the linear problem. As in \cite{BSTU06} after this expansion we are back in the spatially homogeneous 
set-up except that Fourier transform has been replaced by the  Bloch transform. 

\subsection{Spectral properties}
The linearized problem 
\begin{align}
\label{Boussinesqlin}
\partial_t^2 u(x,t) 
=
\partial_x(a(x)\partial_x u(x,t))
-\partial_x^2(b(x)\partial_x^2u(x,t) )
\end{align}
is solved by so called Bloch modes
$$ 
u(x,t)= w(x) e^{ilx} e^{i \omega t},
$$
with $ w $ being $ 2\pi $-periodic w.r.t. $ x $
satisfying 
$$
-(\partial_x+ il)(a(x)(\partial_x+il) w(x))
+(\partial_x+il)^2(b(x)(\partial_x+il)^2 w(x) )
= \omega^2 w(x).
$$
The left hand side defines a self-adjoint elliptic operator 
$ L_{l}(\partial_x):H^{\theta+4}\to H^{\theta}$.
Hence, for fixed $ l$ there exists a countable set of eigenvalues $ \lambda_n(l) $, with $ n \in \mathbb{N}$,
ordered such that $ \lambda_{n+1}(l) \geq \lambda_n(l) $,
with associated eigenfunctions $ w_n(x,l)$.

\begin{lemma}
For $ l = 0 $ the operator $ L_0(\partial_x)$
possesses the simple eigenvalue $ \lambda_1(0) =0 $
associated to the eigenfunction $ \widetilde{w}_1(0,x) = 1$.
\end{lemma}
\noindent
{\bf Proof.}
Obviously we have $ L_0(\partial_x) 1 = 0 $.
Moreover, we have 
$$ 
(w,L_0(\partial_x) w )_{L^2} = \int_{-1/2}^{1/2} a(x) (\partial_x w(x))^2 dx + \int_{-1/2}^{1/2} b(x) (\partial_x^2 w(x))^2 dx \geq 0 .
$$
Hence  $ L_0(\partial_x) w = 0 $ implies $ \partial_x w = 0 $. From the $ 2 \pi $-periodicity it follows $ w = const$. 
Hence $ \lambda_1(0) = 0 $ is a simple eigenvalue. 
\qed
\medskip

It is well known that the curves $ l \mapsto 
\lambda_n(l) $ and $ l \mapsto \widetilde{w}_n(l,\cdot) $
are smooth w.r.t. $ l $ for simple eigenvalues.
Hence, there exists a $ \delta_0 > 0 $ such that 
for $ l \in [-\delta_0,\delta_0] $ the smallest eigenvalue
$  \lambda_1(l) $ is separated from the rest of the spectrum.
Since  $ L_l(\partial_x) $ is self-adjoint and positive-definite for all $ l $ we have 
$  \lambda_1(l) \geq 0  $ for all $ l $.
In the KdV equation only odd and in the Whitham system only even spatial 
derivatives occur. This is a consequence of the following lemma.
\begin{lemma} \label{craven32}
The curve  $  l  \mapsto  \lambda_1(l) $ for $ l \in [-\delta_0,\delta_0] $ is an even real-valued function. The associated eigenfunctions satisfy  $ \widetilde{w}_1(l,x) =  \overline{\widetilde{w}_1(-l,x)} $. 
Under the assumption that the coefficient functions $ a $ and $ b $ are even, the eigenfunctions possess
an expansion 
$$ 
\widetilde{w}_1(l,x) = \sum_{j= 0}^{\infty} (il)^j g_j(x) ,
$$ 
with $ g_0(x) = 1 $,  $ \int_0^{2 \pi} g_j(x) dx = 0 $ for $ j \geq 1 $,
$$ 
g_{2j}(x) = g_{2j}(-x) \in \mathbb{R} \qquad \textrm{and} \qquad g_{2j+1}(x) = - g_{2j+1}(-x)  \in \mathbb{R} .
$$ 
\end{lemma}
{\bf Proof.}
The first two statements follow from the fact that for fixed $ l $ the operator $ L_l(\partial_x) $ is  self-adjoint
and from the fact that \eqref{Boussinesq} is a real problem.  For $ (il)^0  $ we obtain 
$$
-\partial_x(a(x)\partial_x g_0(x))
+\partial_x^2(b(x)\partial_x^2 g_0(x) )
= 0
$$
which is, as we already know, uniquely been solved by $ g_0(x) = 1 $. For $ (il)^1  $ we obtain 
$$
-\partial_x(a(x)\partial_x g_1(x))
+\partial_x^2(b(x)\partial_x^2 g_1(x) )
- \partial_x a(x) = 0 .
$$
The term $ \partial_x a(x) $ is odd. The subspace of odd functions is invariant for the  
differential operator $  L_0(\partial_x) = -\partial_x(a(x)(\partial_x \cdot)
+\partial_x^2(b(x)\partial_x^2 \cdot ) $. Moreover in this subspace its spectrum is bounded away from zero
such that this equation possesses a unique
odd solution $ g_1 = g_1(x) $.
For $ (il)^2  $ we obtain 
$$
-\partial_x(a(x)\partial_x g_2(x))
+\partial_x^2(b(x)\partial_x^2 g_2(x) )  + 1 + f_2(x)
= 1 
$$
with $ f_2(x) $ an even function depending on $ a $, $ b $, $ g_0 $, and $ g_1 $ and possessing vanishing mean value.
In the subspace of vanishing mean value the differential operator $ L_0(\partial_x) $ possesses spectrum which is bounded away from zero
such that this equation possesses a unique
even solution $ g_2 = g_2(x) $. With the same arguments the next  orders with the stated properties can be computed. The convergence of the series in a neighborhood of  $ l = 0 $ 
in  $ H^{\theta} $ for every $ \theta \geq 0 $
follows from the smoothness 
of the curve of simple eigenfunctions w.r.t. $ l $ and the smoothness of the coefficient functions 
$ a $, $ b $, and $ c $ w.r.t. $ x $.
\qed
\medskip

The KdV equation, the inviscid
Burgers equation, and the Whitham system describe the modes associated to the curve $ \lambda_1 $
close to $ l = 0 $. Therefore, in order to derive these amplitude equations we consider the Bloch transform 
$$
u(x,t) = \int_{-1/2}^{1/2} \widetilde{u}(l,x,t)e^{ilx} 
dx
$$
of 
\eqref{Boussinesq}, namely 
\begin{equation} \label{ortloff}
\partial_t^2 \widetilde{u}(l,x,t)= 
- L_l(\partial_x) \widetilde{u}(l,x,t) + N_l(\partial_x)(\widetilde{u})(l,x,t)
\end{equation}
where 
$$
N_l(\partial_x)(\widetilde{u})(l,x,t) = (\partial_x+ il)(c(x)(\partial_x+il) \int_{-1/2}^{1/2}
\widetilde{u}(l-m,x,t) \widetilde{u}(m,x,t) dm.
$$
Then we make the ansatz 
$$
\widetilde{u}(l,x,t) = \chi_{[-\delta_0/2,\delta_0/2]}(l) \widetilde{u}_1(l,t) \widetilde{w}_1(l,x) +\widetilde{v}(l,x,t)
$$
with 
$$
\int_{0}^{2 \pi} \overline{\widetilde{w}_1(l,x)}
\widetilde{v}(l,x,t) dx= 0
$$
for $ l \in [-\delta_0/2,\delta_0/2] $
and find 
\begin{eqnarray*}
\partial_t^2  \widetilde{u}_1(l,t) & = &  - \lambda_1(l) \widetilde{u}_1(l,t)  + P_c(l) N_l(\partial_x)(\widetilde{u})(l,t) ,\\
\partial_t^2 \widetilde{v}(l,x,t)& = & 
- L_l(\partial_x) \widetilde{v}(l,x,t) + P_s(l) N_l(\partial_x)(\widetilde{u})(l,x,t),
\end{eqnarray*}
where 
\begin{eqnarray*}
(P_c \widetilde{u})(l,t) & = & \frac{1}{2\pi}\chi_{[-\delta_0/2,\delta_0/2]}(l) \int_{0}^{2 \pi} \overline{\widetilde{w}_1(l,x)} \widetilde{u}(l,x,t) dx, \\
(P_s \widetilde{u})(l,x,t) & = & \widetilde{u}(l,x,t) - (P_c  \widetilde{u})(l,t){w}_1(l,x).
\end{eqnarray*}
All amplitude equations which we have in mind can be derived in a very similar way.
They describe the evolution of the $ \widetilde{u}_1$
modes which are concentrated in an 
$ \mathcal{O}(\varepsilon)$ neighborhood 
of the Bloch wave number $ l = 0$.
In all three cases we make an ansatz 
\begin{eqnarray} \label{blochansatz}
\widetilde{u}_1(l,t) = \varepsilon^{-1} \varepsilon^{\alpha} \chi_{[-\delta_0/4,\delta_0/4]}(\frac{l}{\varepsilon}) \widehat{A}(\frac{l}{\varepsilon},\varepsilon^{1+ \alpha} t) e^{i l c t}
\end{eqnarray}
with $ \alpha =2 $ and $ c >0  $ for the KdV approximation, $ \alpha \in (0,2) $ and $ c >0 $ for the 
inviscid  Burgers approximation, and $ \alpha =0 $ and $ c =0 $ for the Whitham approximation,
cf. the text below Figure \ref{fig:example}.
The amplitude $  \widehat{A} $ will be defined in Fourier space and the cut-off function $ \chi_{[-\delta_0/4,\delta_0/4]}(\frac{l}{\varepsilon}) $ allows to transfer $  \widehat{A} $ into Bloch space.
In the following we use the abbreviation
\begin{equation} \label{crew1}
\widetilde{A}(\frac{l}{\varepsilon},\varepsilon^{1+ \alpha} t) =  \chi_{[-\delta_0/4,\delta_0/4]}(\frac{l}{\varepsilon}) \widehat{A}(\frac{l}{\varepsilon},\varepsilon^{1+ \alpha} t).
\end{equation}
For each of the three approximations we have to derive 
the associated amplitude equation and  to compute and estimate the residual terms
\begin{eqnarray*}
\textrm{Res}(\widetilde{u})(l,x,t) & = &
- \partial_t^2 \widetilde{u}(l,x,t) -
L_l(\partial_x) \widetilde{u}(l,x,t) + N_l(\partial_x)(\widetilde{u})(l,x,t).
\end{eqnarray*}

\subsection{Derivation of the KdV and the inviscid  Burgers equation}

The amplitude equations which we have in mind have derivatives in front of the nonlinear terms.
Hence before deriving these equations we need to prove a number of properties about the 
nonlinear terms. 
We introduce kernels $ s_{11}^1(l,l-m,m), \ldots ,  s_{vv}^v(l,l-m,m) $ by 
\begin{eqnarray*}
(P_c N_l(\partial_x)(\widetilde{u}))(l,t) & = &  \int_{-1/2}^{1/2} s_{11}^1(l,l-m,m) \widetilde{u}_1(l-m,t)\widetilde{u}_1(m,t)dm \\
&& +  \int_{-1/2}^{1/2} s_{1v}^1(l,l-m,m) \widetilde{u}_1(l-m,t)\widetilde{v}(m,x,t)dm
 \\
&& +  \int_{-1/2}^{1/2} s_{v1}^1(l,l-m,m) \widetilde{v}(l-m,x,t)\widetilde{u}_1(m,t)dm
 \\
&& +  \int_{-1/2}^{1/2} s_{vv}^1(l,l-m,m) \widetilde{v}(l-m,x,t)\widetilde{v}(m,x,t)dm
\end{eqnarray*}
and 
\begin{eqnarray*}
(P_s N_l(\partial_x)(\widetilde{u}))(l,x,t) & = &  \int_{-1/2}^{1/2} s_{11}^v(l,l-m,m) \widetilde{u}_1(l-m,t)\widetilde{u}_1(m,t)dm \\
&& +  \int_{-1/2}^{1/2} s_{1v}^v(l,l-m,m) \widetilde{u}_1(l-m,t)\widetilde{v}(m,x,t)dm
 \\
&& +  \int_{-1/2}^{1/2} s_{v1}^v(l,l-m,m) \widetilde{v}(l-m,x,t)\widetilde{u}_1(m,t)dm
 \\
&& +  \int_{-1/2}^{1/2} s_{vv}^v(l,l-m,m) \widetilde{v}(l-m,x,t)\widetilde{v}(m,x,t)dm.
\end{eqnarray*}

For the derivation of the KdV and the Burgers equation we need
\begin{lemma} \label{lem33}
We have 
$$ 
|s_{11}^1(l,l-m,m) - \nu_2 l^2  | \leq  C |l|(l^2 +  (l-m)^2 + m^2),
$$
where 
\begin{equation} \label{nu2}
\nu_2 = - \frac{1}{2 \pi} \int_{0}^{2 \pi} c(x)  (1+ \partial_x g_1(x) )^2  dx. 
\end{equation}
\end{lemma}
\noindent
{\bf Proof.}
Due to Lemma \ref{craven32} we have 
\begin{equation} \label{gmexpand}
\widetilde{w}_1(l,x) = 1 + il g_1(x) + \mathcal{O}(l^2) 
\end{equation}
where $ g_1(x) \in \mathbb{R} $ with $ \int_0^{2\pi} g_1(x) dx = 0 $.
This expansion yields
\begin{eqnarray*} 
&& 2 \pi s_{11}^1(l,l-m,m) 
\\& = & \int_{0}^{2 \pi}\overline{\widetilde{w}_1(l,x)} (\partial_x+ il)(c(x)(\partial_x+il)
( \widetilde{w}_1(l-m,x)  \widetilde{w}_1(m,x)) dx
\\& = & \int_{0}^{2 \pi} (1 - il g_1(x) + \mathcal{O}(l^2) )
(\partial_x+ il)(c(x)(\partial_x+il) \\ && \qquad  \times ((1 + i(l-m) g_1(x) + \mathcal{O}((l-m)^2) )(1 + im g_1(x) + \mathcal{O}(m^2) )) dx 
\\& = & - \int_{0}^{2 \pi} c(x) ( (\partial_x-  il) (1 - il g_1(x) + \mathcal{O}(l^2) ) )
((\partial_x+il) \\&& \qquad  \times ((1 + i(l-m) g_1(x) + \mathcal{O}((l-m)^2) )(1 + im g_1(x) + \mathcal{O}(m^2) )) dx 
\\& = & - \int_{0}^{2 \pi} c(x)  (-il  - il \partial_x g_1(x) + \mathcal{O}(l^2) )\\ && \qquad \times
((\partial_x+il)  ((1 + i l g_1(x) + \mathcal{O}((l-m)^2 + m^2) ) dx 
\\& = & - \int_{0}^{2 \pi} c(x)  (-il  - il \partial_x g_1(x) + \mathcal{O}(l^2) )\\ && \qquad \times
( il  + i l \partial_x g_1(x) + \mathcal{O}((l-m)^2 + m^2) ) dx 
\\& = & \nu_2 l^2 +  \mathcal{O}( |l| (l^2 +  (l-m)^2 + m^2)).
\end{eqnarray*}
We remark already at this point that due to  the fact that $ a $, $ b $, and $ c $ are assumed to be even we have for symmetry reasons that the 
higher order terms are not only $  \mathcal{O}( |l| (l^2 +  (l-m)^2 + m^2)) $, but  $\mathcal{O}( l^4 +  (l-m)^4 + m^4) $.
See below.\qed 
\bigskip

The following derivation of amplitude equations in Fourier or Bloch space 
is straightforward and documented in various papers. We refer to 
\cite[Chapter 5]{Schn11OW} for an introduction.

\subsubsection{The KdV equation.}

We start with the KdV approximation $ \varepsilon^2 \Psi $ which is defined via 
\eqref{blochansatz} for $ \alpha = 2 $ and which is inserted into $ \textrm{Res}(\widetilde{u}) $. 
We find with $ \widetilde{u}_1(l,t)  = \varepsilon \widetilde{A}(K,T) \mathbf{E} $, 
$ \mathbf{E}  = e^{i \varepsilon K c t} $,
$ T = \varepsilon^{3} t $, and $ l = \varepsilon K $ that 
\begin{eqnarray*} 
P_c (\textrm{Res}(\widetilde{u}))(l,t) & = & 
- \partial_t^2  \widetilde{u}_1(l,t)   - \lambda_1(l) \widetilde{u}_1(l,t)  \\ && + 
\int_{-1/2}^{1/2} s_{11}^1(l,l-m,m) \widetilde{u}_1(l-m,t)\widetilde{u}_1(m,t)dm
\\& = & \varepsilon^3 c^2 K^2 \widetilde{A}(K,T)  \mathbf{E} -  2 \varepsilon^5 i c K (\partial_T \widetilde{A}(K,T) )  \mathbf{E} - \varepsilon^7 (\partial_T^2 \widetilde{A}(K,T) )  \mathbf{E} 
\\&&
-  \varepsilon^3 \lambda_1''(0) K^2 \widetilde{A} (K,T)  \mathbf{E}/2 
 -  \varepsilon^5 \lambda_1''''(0) K^4 \widetilde{A}(K,T)   \mathbf{E}/24 
 + \mathcal{O}( \varepsilon^7)
\\&&
+ \varepsilon^5 \int_{-1/(2  \varepsilon)}^{1/(2  \varepsilon)} \nu_2 K^2 \widetilde{A}(K-M,T)\widetilde{A}(M,T)
dM  \mathbf{E}+ \mathcal{O}( \varepsilon^6).
\end{eqnarray*}
If $ \widehat{A}(\cdot,T) \in L^2_s $ then the error made by replacing 
$ \int_{-1/(2  \varepsilon)}^{1/(2  \varepsilon)} \ldots dM $ by 
$ \int_{- \infty}^{\infty} \ldots dM $ is $ \mathcal{O}(\varepsilon^{s-1/2}) $.
Hence by equating the coefficients of  $  \varepsilon^3 $ and $  \varepsilon^5 $
to zero we find 
$ c^2=  \lambda_1''(0)/2 $ 
and 
$ \widehat{A} $ to satisfy 
$$ 
-  2  i c  \partial_T \widehat{A}(K,T)  - \lambda_1''''(0) K^3 \widehat{A}(K,T)/24 
+  \int_{- \infty}^{\infty}  \nu_2 K \widehat{A}(K-M,T)\widehat{A}(M,T) 
dM = 0,
$$
respectively, $ A $   to satisfy the KdV equation
\begin{equation} \label{main_KdV}
 2  c  \partial_T  A(X,T)  + \lambda_1''''(0) \partial_X^3 A(X,T)/24 +
 \nu_2 \partial_X (A(X,T)^2)= 0.
\end{equation}

\subsubsection{The inviscid Burgers equation}

Due to the explanations  in the Appendix \ref{appA} we restrict to the case $ \alpha = 1 $.
We insert the inviscid Burgers  approximation $ \varepsilon^{\alpha} \Psi $, which is defined via 
\eqref{blochansatz} for $ \alpha = 1 $, into $ \textrm{Res}(\widetilde{u}) $. 
We find with $ \widetilde{u}_1(l,t)  =  \widetilde{A}(K,T) \mathbf{E} $, 
$ \mathbf{E}  = e^{i \varepsilon K c t} $,
$ T = \varepsilon^{2} t $, and $ l = \varepsilon K $ that 
\begin{eqnarray*} 
P_c (\textrm{Res}(\widetilde{u}))(l,t) & = & 
- \partial_t^2  \widetilde{u}_1(l,t)   - \lambda_1(l) \widetilde{u}_1(l,t)  \\�&& + 
\int_{-1/2}^{1/2} s_{11}^1(l,l-m,m) \widetilde{u}_1(l-m,t)\widetilde{u}_1(m,t)dm
\\& = & \varepsilon^2 c^2 K^2 \widetilde{A}(K,T)  \mathbf{E} -  2 \varepsilon^3 i c K (\partial_T \widetilde{A}(K,T) )  \mathbf{E} - \varepsilon^4 (\partial_T^2 \widetilde{A}(K,T) )  \mathbf{E} 
\\&&
-  \varepsilon^2 \lambda_1''(0) K^2 \widetilde{A} (K,T)  \mathbf{E}/2 
 -  \varepsilon^4 \lambda_1''''(0) K^4 \widetilde{A}(K,T)   \mathbf{E}/24 
 + \mathcal{O}( \varepsilon^4)
\\&&
+ \varepsilon^3 \int_{-1/(2  \varepsilon)}^{1/(2  \varepsilon)} \nu_2 K^2 \widetilde{A}(K-M,T)\widetilde{A}(M,T)
dM  \mathbf{E}+ \mathcal{O}( \varepsilon^4).
\end{eqnarray*}
We proceed as above and
equate the coefficients of  $  \varepsilon^2 $ and $  \varepsilon^3 $
to zero. We find 
$ c^2=  \lambda_1''(0)/2 $ 
and 
$ \widehat{A} $ to satisfy 
$$ 
-  2  i c  \partial_T \widehat{A}(K,T)  
+  \int_{- \infty}^{\infty}  \nu_2 K \widehat{A}(K-M,T)\widehat{A}(M,T) 
dM = 0
$$
respectively $ A $   to satisfy the inviscid Burgers equation
\begin{equation} \label{main_Burgers}
 2  c  \partial_T  A(X,T)  +
 \nu_2 \partial_X (A(X,T)^2)= 0.
\end{equation}

\subsection{Derivation of the Whitham system}

The derivation of the Whitham system is much more involved since already in the derivation 
the $ \widetilde{v} $ part has to be included. 
Due to the symmetry assumption {\bf (SYM)} with $ u = u(x,t) $, also $ u = u(-x,t) $
is a solution of  \eqref{Boussinesq}. 
As a consequence in  \eqref{Boussinesq} all terms must contain an even number of $ \partial_x $-derivatives.
Since in Bloch space  
\begin{eqnarray*}
u(-x,t) & = & \int_{-1/2}^{1/2} \widetilde{u}(-x,l) e^{-ilx} dl \\
& = &  - \int_{1/2}^{-1/2} \widetilde{u}(-x,-l) e^{ilx} dl 
 =   \int_{-1/2}^{1/2} \widetilde{u}(-x,-l) e^{ilx} dl 
\end{eqnarray*}
with $ \widetilde{u} = \widetilde{u}(l,x,t) $, also $ \widetilde{u} = \widetilde{u}(-l,-x,t) $ is a 
solution of the Bloch wave transformed system \eqref{ortloff}. As a consequence in  \eqref{ortloff} all terms must contain an even number of $ \partial_x $-derivatives or $ il $, $ i(l-m) $, or $ im $  factors, i.e., for instance $ il \partial_x $ can occur, but $ -l^2 \partial_x $ not.
Before we start with the derivation of the Whitham system we additional need 
that in some  of the  kernel functions $ s_{j_1j_2}^j $ at least one $ l $ factor occurs.
\begin{lemma} \label{lem34}
We have 
$$ 
|s_{vv}^1(l,l-m,m)| \leq  C |l|
$$
and 
$$ 
|s_{11}^v(l,l-m,m)| \leq  C (|l| +  (l-m)^2 + m^2).
$$
\end{lemma}
\noindent
{\bf Proof.}
a) Using again the expansion \eqref{gmexpand} yields after some integration by parts that
\begin{eqnarray*} 
&  & \int_{0}^{2 \pi} \overline{\widetilde{w}_1(l,x)} (\partial_x+ il)(c(x)(\partial_x+il)
 \int_{-1/2}^{1/2}
\widetilde{v}(l-m,x,t) \widetilde{v}(m,x,t)dm dx 
\\& = &\int_{-1/2}^{1/2} \int_{0}^{2 \pi}  c(x)  (-il  + il \partial_x g_1(x) + \mathcal{O}(l^2) ) 
(\partial_x+il)(\widetilde{v}(l-m,x,t) \widetilde{v}(m,x,t)) dx dm
\\& = &   \mathcal{O}(l)
\end{eqnarray*}
b) As above we obtain
\begin{eqnarray*} 
&& s_{11}^v(l,l-m,m) 
\\& = & \int_{0}^{2 \pi}\overline{\widetilde{v}(l,x)} (\partial_x+ il)(c(x)(\partial_x+il)
( \widetilde{w}_1(l-m,x)  \widetilde{w}_1(m,x)) dx
\\& = & \int_{0}^{2 \pi} \overline{\widetilde{v}(l,x)}
(\partial_x+ il)(c(x)(\partial_x+il) \\&& \qquad  \times ((1 + i(l-m) g_1(x) + \mathcal{O}((l-m)^2) )(1 + im g_1(x) + \mathcal{O}(m^2) )) dx 
\\& = &  \int_{0}^{2 \pi} \overline{\widetilde{v}(l,x)} (\partial_x+ il)(c(x)
( il  + i l \partial_x g_1(x) + \mathcal{O}((l-m)^2 + m^2) ) dx 
\\& = & \mathcal{O}(|l| +  (l-m)^2 + m^2).
\end{eqnarray*}
\qed

For the derivation of the Whitham system
we make the ansatz
\begin{eqnarray} \label{blochansatzwhitham}
\widetilde{u}_1(l,t)  =  \varepsilon^{-1}  \widetilde{A}(K,T) 
\qquad \textrm{and} \qquad
\widetilde{v}(l,x,t)  =   \widetilde{B}(K,x,T).
\end{eqnarray}
where $ T = \varepsilon t $, and $ l = \varepsilon K $. 
With $ \widetilde{u}(l,x,t) = \widetilde{u}_1(l,t) \widetilde{w}_1(l,x) +\widetilde{v}(l,x,t) $ we find  that 
\begin{eqnarray*} 
P_c (\textrm{Res}(\widetilde{u}))(l,t) & = & 
- \partial_t^2  \widetilde{u}_1(l,t)   - \lambda_1(l) \widetilde{u}_1(l,t) + P_c(l) N_{l}(\partial_x)(\widetilde{u})(l,t)  
\\
& = & - \varepsilon \partial_T^2 \widetilde{A}(K,T)  -  \varepsilon \lambda_1''(0) K^2 
\widetilde{A}(K,T)  /2 
+ \mathcal{O}(\varepsilon^3)
\\ && \qquad+ P_c(\varepsilon K) N_{\varepsilon K}(\partial_x)(\widetilde{u})( \varepsilon K,T/\varepsilon)  
\end{eqnarray*}
and 
\begin{eqnarray*} 
P_s (\textrm{Res}(\widetilde{u}))(l,x,t) 
& = & 
- \partial_t^2 \widetilde{v}(l,x,t)
- L_l(\partial_x) \widetilde{v}(l,x,t) + P_s(l) N_l(\partial_x)(\widetilde{u})(l,x,t) \\
& = & 
-\eps^2\p_T^2 \widetilde{B}(K,x,T) -\widetilde{L}_{\varepsilon K}(\partial_x)\widetilde{B}(K,x,T) 
\\ && \qquad +P_s(l) N_{\varepsilon K}(\partial_x)(\widetilde{u})(\varepsilon K,x,T/\varepsilon).
\end{eqnarray*}
Since $ P_s(\varepsilon K) N_{\varepsilon K}(\partial_x)(\widetilde{u}) $ is quadratic w.r.t.  $  \widetilde{u} $
and since $ \widetilde{L}_{\varepsilon K} $ is invertible on the range of $ P_s(\varepsilon K) $ we can use the implicit
function theorem to solve 
$$
- \widetilde{L}_{\varepsilon K}(\partial_x)\widetilde{B}(K,x,T) 
+P_s(\varepsilon K) N_{\varepsilon K}(\partial_x)(\widetilde{u})(\varepsilon K,x,T/\varepsilon) = 0 
$$ 
w.r.t. $ \widetilde{B} = H(\widetilde{A})(K,x,T)  $ for sufficiently small $ \widetilde{A} $.
Note that we kept our notation and still wrote $ T/\varepsilon $ in the arguments of $ N $ 
although in fact it only depends on $ T $.
We insert  $ \widetilde{B} = H(\widetilde{A})(K,x,T)  $ into the first equation and obtain 
\begin{eqnarray*} 
P_c (\textrm{Res}(\widetilde{u}))(l,t) & = & 
 - \varepsilon \partial_T^2 \widetilde{A}(K,T)  -  \varepsilon \lambda_1''(0) K^2  
\widetilde{A}(K,T)  /2 + \mathcal{O}(\varepsilon^3) \\ && + P_c(\varepsilon K) N_{\varepsilon K}(\partial_x)( \varepsilon^{-1}  \widetilde{A}(K,T)\widetilde{w}_1(\varepsilon K,x) 
\\ && \qquad + H(\widetilde{A}) (K,x,T)  )(\varepsilon K,T/\varepsilon)  
\end{eqnarray*}
The Whitham system occurs by expanding  
the right hand side w.r.t. $ \varepsilon $
and by equating the coefficient in front of $ \varepsilon^1 $ to zero.
We obtain in a first step 
$$ 
 \partial_T^2 \widetilde{A}(K,T)  +   \lambda_1''(0) K^2 
\widetilde{A}(K,T)  /2 + \widetilde{G}(\widetilde{A})(K,T) = 0 
$$ 
where $  \widetilde{G} $ is a nonlinear function that can be written as 
$$ 
 \widetilde{G}(\widetilde{A})(K,T) = - \chi_{[-\delta_0/4,\delta_0/4 ]}(\varepsilon K)\sum_{j=2}^{\infty} s_j iK \int_{-1/(2\varepsilon)}^{1/(2\varepsilon)}\widetilde{A}^{*(j-1)}(K-M) iM \widetilde{A}(M) dM
$$
with coefficients $ s_j $.
The factor $ iK $ comes from Lemma \ref{lem33} and Lemma \ref{lem34} a), the factor $ iM $ from the fact that due to the reflection symmetry we need an even number of such factors 
and due to the long wave character of the approximation we have exactly two such factors at $ \varepsilon $.   
Replacing via \eqref{crew1} the Bloch transform $ \widetilde{A}(K,T) $ by the Fourier transform $ \widetilde{A}(K,T) $ finally gives Whitham's system
\begin{equation} \label{crew2}
 \partial_T^2 \widehat{A}(K,T)  +   \lambda_1''(0) K^2 
\widehat{A}(K,T)  /2 + \widehat{G}(\widehat{A})(K,T) = 0 
\end{equation}
in Fourier space
where $  \widehat{G} $ is a nonlinear function that can be written as 
$$ 
 \widehat{G}(\widehat{A})(K,T) = - \sum_{j=2}^{\infty} s_j iK \int_{-\infty}^{\infty}\widehat{A}^{*(j-1)}(K-M) iM \widehat{A}(M) dM.
$$
In physical space we have 
$$ 
G(A)(X,T) = -\sum_{j=2}^{\infty} s_j \partial_X ( A^{j-1} \partial_X A) = -\partial_X^2 
 \sum_{j=2}^{\infty} s_j A^j /j  
$$ 
such that Whitham's system finally can be written as 
\begin{equation} \label{crew3}
\partial_T^2 A = \partial_X^2 \mathcal{H}(A) , \qquad \textrm{with} \qquad 
\mathcal{H}(A) =  - \lambda_1''(0) A /2 -  \sum_{j=2}^{\infty} s_j A^j /j  .
\end{equation}

\section{Estimates for the residual}

After the derivation of the  amplitude equations we estimate the so called residual, the terms which do not cancel
after inserting the approximation into   \eqref{Boussinesq}. In order to have estimates as in the spatially homogeneous case 
for the residual terms in terms of $ \varepsilon $ we have to modify our approximations with higher order terms.
\medskip

\noindent
{\bf The improved KdV approximation.}
For the construction of the improved KdV approximation we proceed as for  the 
derivation of the Whitham system.
With    $ \mathbf{E}  = e^{i \varepsilon K c t} $,
$ T = \varepsilon^{3} t $, and $ l = \varepsilon K $
we make the ansatz
\begin{eqnarray*} 
\widetilde{u}_1(l,t)  & = & \varepsilon \widetilde{A}(K,T) \mathbf{E} \\
\widetilde{v}(l,x,t)  & = &   \varepsilon^4  \widetilde{B}(K,x,T)\mathbf{E} +  \varepsilon^5  \widetilde{B}_2(K,x,T)\mathbf{E} +  \varepsilon^3  \widetilde{B}_3(K,x,T)\mathbf{E} .
\end{eqnarray*}
With $ \widetilde{u}(l,x,t) = \widetilde{u}_1(l,t) \widetilde{w}_1(l,x) +\widetilde{v}(l,x,t) $,
$ T = \varepsilon t $, and $ l = \varepsilon K $ we find  that 
\begin{eqnarray*} 
P_c (\textrm{Res}(\widetilde{u}))(l,t) & = & 
- \partial_t^2  \widetilde{u}_1(l,t)   - \lambda_1(l) \widetilde{u}_1(l,t) + P_c(l) N_{l}(\partial_x)(\widetilde{u})(l,t)  
\\ & = &  \varepsilon^3 c^2 K^2 \widetilde{A}(K,T)  \mathbf{E} -  2 \varepsilon^5 i c K (\partial_T \widetilde{A}(K,T) )  \mathbf{E} - \varepsilon^7 (\partial_T^2 \widetilde{A}(K,T) )  \mathbf{E} 
\\&&
-  \varepsilon^3 \lambda_1''(0) K^2 \widetilde{A} (K,T)  \mathbf{E}/2 
 -  \varepsilon^5 \lambda_1''''(0) K^4 \widetilde{A}(K,T)   \mathbf{E}/24 
 + \mathcal{O}( \varepsilon^7)
\\&&
+  \varepsilon^5 \int_{-1/(2  \varepsilon)}^{1/(2  \varepsilon)} \nu_2 K^2 \widetilde{A}(K-M,T)\widetilde{A}(M,T)
dM  \mathbf{E}+ \mathcal{O}( \varepsilon^7) =\mathcal{O}( \varepsilon^7)  
\end{eqnarray*}
if we choose $ c $ and $ \widetilde{A} $ as above. We have $ \mathcal{O}( \varepsilon^7)   $ and not $ \mathcal{O}( \varepsilon^6) $ since   
$ P_c (\textrm{Res}(\widetilde{u}))(l,t) $ does not depend on $ x $ and has to be even w.r.t. factors in $ l $, i.e., $  \varepsilon^5 K^4 \widetilde{A}(K,T)  $ is allowed,
but not  $  \varepsilon^6 K^5 \widetilde{A}(K,T)  $.
Next we have 
\begin{eqnarray*} 
P_s (\textrm{Res}(\widetilde{u}))(l,x,t) 
& = & - \partial_t^2 \widetilde{v}(l,x,t) 
- L_l(\partial_x) \widetilde{v}(l,x,t) + P_s(l) N_l(\partial_x)(\widetilde{u})(l,x,t) \\
& = & 
 c^2 K^2  (\varepsilon^6 \widetilde{B}(K,x,T) + \varepsilon^7 \widetilde{B}_2(K,x,T)  + \varepsilon^{8} \widetilde{B}_3(K,x,T) ) \mathbf{E} 
 \\ 
&& -  2 i c K (\varepsilon^8  \partial_T \widetilde{B}(K,x,T) +\varepsilon^{9}  \partial_T \widetilde{B}_2(K,x,T)+\varepsilon^{10}  \partial_T \widetilde{B}_3(K,x,T))  \mathbf{E}  
\\ &&  -( \varepsilon^{10} \partial_T^2 \widetilde{B}(K,x,T) + \varepsilon^{11} \partial_T^2 \widetilde{B}_2(K,x,T)+  \varepsilon^{12} \partial_T^2 \widetilde{B}_3(K,x,T)  )  \mathbf{E} 
\\ &&  - (\varepsilon^4 \widetilde{L}_{\varepsilon K}(\partial_x)\widetilde{B}(K,x,T)+ \varepsilon^5 \widetilde{L}_{\varepsilon K}(\partial_x)\widetilde{B}_2(K,x,T) 
\\ && + \varepsilon^6 \widetilde{L}_{\varepsilon K}(\partial_x)\widetilde{B}_3(K,x,T)  )  \mathbf{E}+P_s(\varepsilon K) N_{\varepsilon K}(\partial_x)(\widetilde{u})(\varepsilon K,x,T/\varepsilon)
\end{eqnarray*}
where we expand 
$$ 
P_s(l) N_{\varepsilon K}(\partial_x)(\widetilde{u})(\varepsilon K,x,T/\varepsilon) = (\varepsilon^4 F_4(\widetilde{A})+ \varepsilon^5 F_5(\widetilde{A})+ \varepsilon^6 F_6(\widetilde{A}, \widetilde{B})+ \mathcal{O}(\varepsilon^7)) (K,x,T)  \mathbf{E} .
$$ 
If we set 
\begin{eqnarray*} 
0 & = &  -\widetilde{L}_{\varepsilon K}(\partial_x)\widetilde{B}(K,x,T)   + \varepsilon^4 F_4(\widetilde{A})(K,x,T) , \\ 
0 & = &  -\widetilde{L}_{\varepsilon K}(\partial_x)\widetilde{B}_2(K,x,T)+ \varepsilon^4 F_5(\widetilde{A})(K,x,T) , \\
0 & = & -\widetilde{L}_{\varepsilon K}(\partial_x)\widetilde{B}_3(K,x,T) + \varepsilon^4 F_6(\widetilde{A}, \widetilde{B})(K,x,T) + c^2 K^2  \widetilde{B}(K,x,T)  ,\\
\end{eqnarray*}
we finally have 
$$ 
P_s (\textrm{Res}(\widetilde{u}))(l,x,t)  =  \mathcal{O}(\varepsilon^7).
$$
The functions $ \widetilde{B} $, $ \widetilde{B}_2 $, and $ \widetilde{B}_3 $ are well-defined since  $ \widetilde{L}_{\varepsilon K} $ can be inverted on the 
range of $ P_s(\varepsilon K) $.
\medskip

\noindent
{\bf The improved inciscid Burgers  approximation.}
We leave this part to the reader. We refer to Appendix \ref{appA} where the modified approximation is discussed 
for the spatially homogeneous situation.
\medskip

\noindent
{\bf The improved Whitham  approximation.} We need the residual formally to be of order $ \mathcal{O}(\varepsilon^3) $.
With the previous approximation we already have $ \mathcal{O}(\varepsilon^3) $ for the $ P_c $-part of the residual again due to 
symmetry reasons, but we only have  $ \mathcal{O}(\varepsilon^2) $ for the $ P_s $-part. As above we modify our ansatz into 
\begin{eqnarray*} 
\widetilde{u}_1(l,t)  =  \varepsilon^{-1}  \widetilde{A}(K,T) 
\qquad \textrm{and} \qquad
\widetilde{v}(l,x,t)  =   \widetilde{B}(K,x,T) + \varepsilon^{2}  \widetilde{B}_2(K,x,T) .
\end{eqnarray*}
We define $  \widetilde{A} $ and $ \widetilde{B} $ exactly as above and $ \widetilde{B}_2 $ as solution of 
$$ 
-\partial_T^2 \widetilde{B}(K,x,T)- \widetilde{L}_{\varepsilon K}(\partial_x)\widetilde{B}_2(K,x,T)  = 0 
$$
which is again well-defined due the fact that  $ \widetilde{L}_{\varepsilon K} $ can be inverted on the 
range of $ P_s(\varepsilon K) $.
\medskip

For all three approximations we gain a factor $ \varepsilon^{1/2} $ when we estimate the error in $ L^2 $-based spaces due to the scaling properties of the $ L^2 $ norm. 
Since the error made by the various approximations will be estimated in physical space via energy estimates we conclude for the KdV approximation, 
for the inviscid Burgers approximation, and for the Whitham approximation that 
\begin{lemma} \label{lemkdvres}
Let $A  \in C([0,T_0],H^{6})$ 
be  a solution  of the  KdV equation \eqref{main_KdV}. 
 Then there exist $ \varepsilon_0 > 0 $, $ C_{\rm res} $ such that for all $  \varepsilon \in (0,\varepsilon_0) $ we have 
$$ 
\sup_{t \in [0,T_0/\varepsilon^3]} \|    \mathrm{Res}(   \varepsilon^2 \Psi(\cdot,t,\varepsilon)    )    \|_{H^1} 
\leq C_{\rm res}  \varepsilon^{15/2}
$$
and 
$$ 
\sup_{t \in [0,T_0/\varepsilon^3]} \|  \partial_x^{-1}  \mathrm{Res}(   \varepsilon^2 \Psi(\cdot,t,\varepsilon)    )    \|_{H^1} 
\leq C_{\rm res}  \varepsilon^{13/2}.
$$
\end{lemma}
\begin{lemma}
Let $ \alpha = 1 $ and let $A  \in C([0,T_0],H^{4})$ 
be  a solution  of the  inviscid Burgers  equation \eqref{constkdv1}. 
 Then there exist $ \varepsilon_0 > 0 $, $ C_{\rm res} $ such that for all $  \varepsilon \in (0,\varepsilon_0) $ we have 
$$ 
\sup_{t \in [0,T_0/\varepsilon^{1+\alpha}]} \|    \mathrm{Res}(  \varepsilon^{\alpha}   \Psi(\cdot,t,\varepsilon)    )    \|_{H^1} 
\leq C_{\rm res}  \varepsilon^{(7+4\alpha)/2}
$$
and 
$$ 
\sup_{t \in [0,T_0/\varepsilon^{1+\alpha}]} \|  \partial_x^{-1}  \mathrm{Res}( \varepsilon^{\alpha}    \Psi(\cdot,t,\varepsilon)    )    \|_{H^1} 
\leq C_{\rm res}  \varepsilon^{(5+4\alpha)/2}.
$$
\end{lemma}
\begin{lemma} \label{lemkdvres1}
Let $A  \in C([0,T_0],H^{4})$ 
be  a solution  of the  Whitham equation \eqref{constwhit1}. 
 Then there exist $ \varepsilon_0 > 0 $, $ C_{\rm res} $ such that for all $  \varepsilon \in (0,\varepsilon_0) $ we have 
$$ 
\sup_{t \in [0,T_0/\varepsilon]} \|    \mathrm{Res}(   \Psi(\cdot,t,\varepsilon)    )    \|_{H^1} 
\leq C_{\rm res}  \varepsilon^{7/2}
$$
and 
$$ 
\sup_{t \in [0,T_0/\varepsilon]} \|  \partial_x^{-1}  \mathrm{Res}(   \Psi(\cdot,t,\varepsilon)    )    \|_{H^1} 
\leq C_{\rm res}  \varepsilon^{5/2}.
$$
\end{lemma}

\section{The error estimates}
\label{sec9}

As for spatially homogeneous case the proofs given for the KdV approximations transfer more or less line for line into proofs for the justification of the inviscid Burgers equation and of the Whitham system. 
Our  approximation results  are as follows
\begin{theorem}\label{main_theorem}
Let  $A\in C([0,T_0],H^{6}(\mathbb R)) $  
be a solution of the  KdV equation  \eqref{main_KdV}.
Then there exist $\varepsilon_0>0$, $C>0$ such that for all $\varepsilon \in (0,\varepsilon_0)$ we have  
solutions $ u \in C([0,T_0/\varepsilon^3],H^2)$ of the spatially periodic Boussinesq model \eqref{Boussinesq}
with
\begin{align*}
\sup_{t\in [0,T_0/\varepsilon^3]}\|u(\cdot,t)-\varepsilon^2 A(\varepsilon(\cdot-t),\varepsilon^3t)\|_{H^{2}}\leq C\varepsilon^{5/2}.
\end{align*}
\end{theorem}
\begin{theorem}\label{main_theoremburgers}
Let $ \alpha = 1 $ and let  $A\in C([0,T_0],H^{4}(\mathbb R)) $  
be a solution of the  inviscid Burgers  equation  \eqref{main_KdV}.
Then there exist $\varepsilon_0>0$, $C>0$ such that for all $\varepsilon \in (0,\varepsilon_0)$ we have  
solutions $ u \in C([0,T_0/\varepsilon^3],H^2)$ of the spatially periodic Boussinesq model \eqref{Boussinesq}
with
\begin{align*}
\sup_{t\in [0,T_0/\varepsilon^{1+\alpha}]}\|u(\cdot,t)-\varepsilon^{\alpha} A(\varepsilon(\cdot-t),\varepsilon^{1+\alpha}t)\|_{H^{2}}\leq C\varepsilon^{(1+2\alpha)/2}.
\end{align*}
\end{theorem}
\begin{theorem}\label{main_theoremwhitham}
There exists a $ C_1 > 0 $ such that the 
following holds.  Let $U \in C([0,T_0],H^{4})$ be a solution of the Whitham system  \eqref{constwhit1}
with 
$$ \sup_{T \in [0,T_0]} \| U(\cdot,T) \|_{H^{4}} \leq C_1 .$$
 Then there exist $ \varepsilon_0 > 0 $ and $ C > 0 $ such that
for all $ \varepsilon \in (0,\varepsilon_0) $ we have solutions  
$ u \in C^0([0,T_0/\varepsilon^3],H^2)$ of our spatially periodic Boussinesq model (\ref{Boussinesq}), such that
\begin{align*}
\sup_{t\in [0,T_0/\varepsilon^3]}\|u(\cdot,t)-U(\varepsilon \cdot,\varepsilon t)\|_{H^2}\leq C_2\varepsilon^{1/2}.
\end{align*}
\end{theorem}
{\bf Proof of the Theorems \ref{main_theorem}-\ref{main_theoremwhitham}.}
Since we already have the estimates for the residuals in the Lemmas
\ref{lemkdvres}-\ref{lemkdvres1}
from this point on the remaining estimates can be handled exactly the same. 
The case $ \alpha = 0 $ corresponds to the Whitham approximation and the case $ \alpha = 2 $ to the KdV approximation. 

The difference $ \varepsilon^{(3+2 \alpha)/2} R = u - \varepsilon^{\alpha} \Psi $
 satisfies 
\begin{eqnarray} \label{hagen3}
 \partial^2_tR  & = & \partial_x(a\partial_xR)-\partial_x^2(b\partial_x^2R)+2\partial_x(c\partial_x(\varepsilon^{\alpha}\Psi R)) \\ && \qquad +\varepsilon^{(3+2\alpha)/2}\partial_x(c\partial_x(R^2))+\varepsilon^{-(3+2\alpha)/2}\text{Res}(\varepsilon^{\alpha}\Psi).  \nonumber
 \end{eqnarray}
The first three terms on the right hand side can be  written as  
 $$
 \partial_x(a\partial_xR)-\partial_x^2(b\partial_x^2R)+2\partial_x(c\varepsilon^{\alpha}\Psi\partial_xR)+2\partial_x(c(\partial_x\varepsilon^{\alpha}\Psi) R)
 $$
 The last term is of order $ \mathcal{O}(\varepsilon^{1+\alpha}) $ due to the long wave character
 of the approximation $\varepsilon^{\alpha}\Psi$. More essential  the first three terms can be
 written as 
 $\partial_x(B(\partial_xR)) $ 
where $ B $ is the self-adjoint operator
\begin{align*}
 B =(a+2c\varepsilon^{\alpha}\Psi) -\partial_x(b \partial_x).
\end{align*}
In case $ \alpha > 0 $ 
for sufficiently small $\varepsilon> 0 $ 
and in case $ \alpha =  0 $ for sufficiently small $ \| \Psi \|_{C^0_b} $ 
the linear operator $ B $ is positive definite.  Hence there exists   a positive-definite self-adjoint operator $\mathcal{A}$ with $\mathcal{A}^2=B$.
The associated operator  norm $\|\cdot\|_{\mathcal{A}} = \| \mathcal{A} \cdot \|_{L^2} $ 
is then equivalent to the $H^1$-norm and $\mathcal{A}^{-1}$ is a bounded operator from $L^2 $ 
to $ H^1$.
Hence the equation for the error can be written as 
\begin{eqnarray}\label{error}
\partial^2_tR  & = &\partial_x (\mathcal{A}^2(\partial_xR))+2\partial_x(c(\partial_x\varepsilon^{\alpha}\Psi) R)
\\ && \qquad  +\varepsilon^{(3+2\alpha)/2}\partial_x(c\partial_x(R^2))
+\varepsilon^{-(3+2\alpha)/2}\text{Res}(\varepsilon^{\alpha}\Psi). \nonumber
\end{eqnarray}
In order to bound the solutions of \eqref{error} we use energy estimates.
Therefore, we first multiply \eqref{error}
 with $\partial_tR$ and  integrate the obtained expression w.r.t. $x$.
We obtain
$$
\int (\partial_tR) \partial^2_tR  dx   =  \partial_t \int (\partial_tR)^2 dx/2
$$
and 
\begin{eqnarray*}
&& \int (\partial_tR)  \partial_x (\mathcal{A}^2(\partial_xR)) dx \\�& = & -\int  (\partial_t \partial_x R) (\mathcal{A}^2 
(\partial_xR)) dx = -\int  (\mathcal{A} \partial_t \partial_x R) (\mathcal{A}  \partial_xR) dx 
\\& = &  -\int  ( \partial_t (\mathcal{A} \partial_x R)) (\mathcal{A}  \partial_xR) dx - \int  ( [\partial_t, \mathcal{A}] \partial_x R) (\mathcal{A}  \partial_xR) dx
\\& = &  - \partial_t  \int  (\mathcal{A} \partial_x R)^2 dx /2
- \int  ( [\partial_t, \mathcal{A}] \partial_x R) (\mathcal{A}  \partial_xR) dx.
\end{eqnarray*}
 where
 $$ 
 [\partial_t, \mathcal{A}] \cdot  = \partial_t (\mathcal{A} \cdot) - \mathcal{A} \partial_t \cdot
 $$ 
 is the commutator of the operators $ \mathcal{A} $ and $ \partial_t $. Moreover, we estimate 
 \begin{eqnarray*}
|\int (\partial_tR) 2\partial_x(c(\partial_x\varepsilon^{\alpha}\Psi) R)dx |  & \leq &  C \varepsilon^{1+\alpha} \| \partial_tR\|_{L^2} \| R \|_{H^1}, \\
|\int (\partial_tR) \varepsilon^{(3+ 2 \alpha)/2}\partial_x(c\partial_x(R^2))dx |  & \leq & C \varepsilon^{(3+ 2 \alpha)/2} \| \partial_tR\|_{L^2} \| R \|_{H^2}^2, \\
 |\int (\partial_tR) \varepsilon^{-(3+ 2 \alpha)/2}\text{Res}(\varepsilon^{\alpha} \Psi)dx |  & \leq & C \varepsilon^{2+\alpha} \| \partial_tR\|_{L^2}
  \end{eqnarray*}
 where we used the Lemmas \ref{lemkdvres}-\ref{lemkdvres1}.
 Finally we have 
$$ 
[\partial_t, \mathcal{A}] \partial_x R = (\partial_t \mathcal{A}) \partial_x R
$$ 
such that  
\begin{eqnarray*}
\int  ( [\partial_t, \mathcal{A}] \partial_x R) (\mathcal{A}  \partial_xR) dx & = & \int 
((\partial_t \mathcal{A}) \partial_x R) (\mathcal{A}  \partial_xR) dx.
\end{eqnarray*}
In order to control this term we first note that
\begin{align*}
 (\partial_t \mathcal{A}) \mathcal{A} + \mathcal{A}\partial_t\mathcal{A}=  \partial_t (\mathcal{A}^2)= 2c\partial_t(\varepsilon^{\alpha}\Psi)
\end{align*}
and 
$$
((\partial_t \mathcal{A}) u,v)_{L^2} = ( u,(\partial_t \mathcal{A}) v)_{L^2}
$$
which follows from differentiating the associated formula for $ \mathcal{A} $ w.r.t. $ t $ 
such that 
\begin{align*}
|2\int ((\partial_t\mathcal{A})\partial_xR)(\mathcal{A}\partial_xR) dx| =& 
|\int ( \mathcal{A}(\partial_t\mathcal{A})\partial_xR) \partial_xR+\partial_xR (\partial_t\mathcal{A}( \mathcal{A}\partial_xR ))dx|\\
                                                   =& |\int \partial_xR ( \mathcal{A}  \partial_t\mathcal{A}+(\partial_t\mathcal{A} )\mathcal{A})\partial_xRdx|\\
                                                   =& |\int 2c (\partial_t(\varepsilon^{\alpha}\Psi))(\partial_xR)^2dx|\\
                                                   \leq&2\sup_{x\in\mathbb R}|c(x)\partial_t(\varepsilon^{\alpha}\Psi(x,t))|\|\partial_xR\|^2_{L^2} = \mathcal{O}(\varepsilon^{1+\alpha}) \|\partial_xR\|^2_{L^2}.
\end{align*}
In order to get a bound for the $ L^2 $-norm of $ R $ and not only of its derivatives 
we secondly multiply ''$\partial_x^{-1}$\eqref{error}''
 with $\mathcal{A}^{-2} \partial_x^{-1}\partial_tR$ and  integrate the expression obtained  in this way w.r.t. $x$.
   We find
\begin{eqnarray*}
\int (\mathcal{A}^{-2} \partial_x^{-1} \partial_tR) \partial_x^{-1} \partial^2_tR  dx 
& = & \int (\mathcal{A}^{-1} \partial_x^{-1} \partial_tR) \mathcal{A}^{-1}\partial_t\partial_x^{-1} \partial_tR  dx 
\\
&  = & \partial_t \int (\mathcal{A}^{-1} \partial_x^{-1}\partial_tR)^2 dx/2 \\ && 
- \int (\mathcal{A}^{-1} \partial_x^{-1} \partial_tR) [\partial_t,\mathcal{A}^{-1}]\partial_x^{-1} \partial_tR  dx 
,\\
\int (\mathcal{A}^{-2} \partial_x^{-1} \partial_tR) \partial_x^{-1} \partial_x\mathcal{A}^2\partial_xR dx & = &   - \partial_t  \int  R^2 dx /2.
\end{eqnarray*}
Moreover, using $ \mathcal{A}^{-1}:L^2\to H^1 $ and the self-adjointness of $ \mathcal{A}^{-1} $ we estimate 
 \begin{eqnarray*}
|\int (\mathcal{A}^{-2} \partial_x^{-1} \partial_tR) 2\partial_x^{-1} \partial_x(c(\partial_x\varepsilon^{\alpha}\Psi) R)dx |   & = & |\int (\mathcal{A}^{-1} \partial_x^{-1} \partial_tR) 2\mathcal{A}^{-1} (c(\partial_x\varepsilon^{\alpha}\Psi) R)dx |
\\
& \leq &  C \varepsilon^{1+\alpha} \| \partial_x^{-1} \partial_tR\|_{L^2} \| R \|_{L^2}, \\
|\int (\mathcal{A}^{-2} \partial_x^{-1} \partial_tR) \varepsilon^{(3+2 \alpha)/2}\partial_x^{-1} \partial_x(c\partial_x(R^2))dx | 
& = & 
|\int (\mathcal{A}^{-1} \partial_x^{-1} \partial_tR) \varepsilon^{(3+2 \alpha)/2}\mathcal{A}^{-1} (c\partial_x(R^2))dx | 
 \\ & \leq & C \varepsilon^{(3+2 \alpha)/2} \|\partial_x^{-1}  \partial_tR\|_{L^2} \| R \|_{H^1}^2, \\
 |\int (\mathcal{A}^{-2} \partial_x^{-1}\partial_tR) \varepsilon^{-(3+2 \alpha)/2}\partial_x^{-1}\text{Res}(\varepsilon^{\alpha}\Psi)dx |  & \leq & C \varepsilon^{1+\alpha} \|  \partial_x^{-1}\partial_tR\|_{L^2}. 
  \end{eqnarray*}
 where we used again the Lemmas \ref{lemkdvres}-\ref{lemkdvres1}. 
Finally we have 
$$ 
[\partial_t, \mathcal{A}^{-1}] \partial_x R = (\partial_t \mathcal{A}^{-1}) \partial_x R
$$ 
such that  
\begin{eqnarray*}
 \int (\mathcal{A}^{-1} \partial_x^{-1} \partial_tR) [\partial_t,\mathcal{A}^{-1}]\partial_x^{-1} \partial_tR  dx 
 & = & \int 
(\mathcal{A}^{-1} \partial_x^{-1}\partial_t R)  (\partial_t \mathcal{A}^{-1})\partial_x^{-1} \partial_tR dx.
\end{eqnarray*}
 We write this as half of 
\begin{eqnarray*}   
&& \int 
 ((\partial_t \mathcal{A}^{-1})\partial_x^{-1} \partial_tR) (\mathcal{A}^{-1} \partial_x^{-1}\partial_t R) dx
+ 
\int 
(\mathcal{A}^{-1} \partial_x^{-1}\partial_t R)  (\partial_t \mathcal{A}^{-1})\partial_x^{-1} \partial_tR dx
\\�& = & 
\int 
( \partial_x^{-1}\partial_t R) ( (\partial_t \mathcal{A}^{-1})\mathcal{A}^{-1} + \mathcal{A}^{-1} \partial_t \mathcal{A}^{-1})\partial_x^{-1} \partial_tR dx
\\�& = & 
\int 
( \partial_x^{-1}\partial_t R)  (\partial_t (\mathcal{A}^{-2}))\partial_x^{-1} \partial_tR dx =: s_1
\end{eqnarray*}
From
$$ 
 \partial_t(\mathcal{A}^2 \mathcal{A}^{-2}) = ( \partial_t (\mathcal{A}^2))\mathcal{A}^{-2} + \mathcal{A}^2 \partial_t(\mathcal{A}^{-2}) = 0 
$$ 
it follows that 
$$ 
 \partial_t(\mathcal{A}^{-2}) = -\mathcal{A}^{-2}( \partial_t (\mathcal{A}^2))\mathcal{A}^{-2}= -\mathcal{A}^{-2} (\partial_t(\varepsilon^{\alpha} \Psi)) \mathcal{A}^{-2} 
$$ 
such that 
$$ 
s_1 = \int 
(\mathcal{A}^{-2} \partial_x^{-1}\partial_t R) (\partial_t(\varepsilon^{\alpha} \Psi))(\mathcal{A}^{-2} \partial_x^{-1}\partial_t R) dx 
= \mathcal{O}(\varepsilon^{1+\alpha})\| \mathcal{A}^{-2} \partial_x^{-1}\partial_t R \|_{L^2}^2
$$ 
which can be bounded by $  \mathcal{O}(\varepsilon^{1+\alpha})\| \mathcal{A}^{-1} \partial_x^{-1}\partial_t R \|_{L^2}^2 $. 
If we define 
\begin{align*}
 E(t)\,\,=\,\, \frac 1 2 \left(||\partial_tR||_{L^2}^2\,+||\mathcal{A}^{-1}\partial_x^{-1}\partial_tR||_{L^2}^2\,+\,||R||_{L^2}^2\,+\,||\mathcal{A}\partial_xR||_{L^2}^2\right).
\end{align*}
we find 
\begin{eqnarray*}
\frac{d}{dt}{E} & \leq & C_1 \varepsilon^{1+\alpha} E + C_2 \varepsilon^{(3+2\alpha)/2} E^{3/2}  + C_3 \varepsilon^{1+\alpha} E^{1/2} \\
& \leq & C_1 \varepsilon^{1+\alpha} E + C_2 \varepsilon^{(3+2\alpha)/2} E^{3/2}  + C_3 \varepsilon^{1+\alpha} 
+  C_3 \varepsilon^{1+\alpha} 
E,
\end{eqnarray*}
with constants $ C_1 $, $ C_2 $, and $ C_3 $ independent of $ 0 < \varepsilon \ll 1 $
since all the 
 $\|\partial_tR\|_{L^2} $, $ \|\mathcal{A}^{-1}\partial_x^{-1}\partial_tR\|_{L^2} $, etc. appearing above 
 can be estimated by $ E^{1/2} $.
Choosing  $ \varepsilon^{1/2} E^{1/2} \leq 1 $ gives 
$$ 
\frac{d}{dt}{E}(t)  \leq (C_1+ C_2+C_3) \varepsilon^{1+\alpha} E + C_3 \varepsilon^{1+\alpha} 
$$  
 which can be estimated with Gronwall's inequality and yields 
$$ 
E(t) \leq C_3 T_0 e^{(C_1+ C_2+C_3)  T_0} =: M
 $$ 
for all $ 0 \leq   \varepsilon^{1+\alpha} t \leq T_0 $.
Choosing $ \varepsilon_0  > 0 $ so small that  $ \varepsilon_0^{1/2} M^{1/2} \leq 1 $
gives the required estimate first for $ E(t) $.
Since in case $ \alpha > 0 $ 
for sufficiently small $\varepsilon> 0 $ 
and case $ \alpha =  0 $ for sufficiently small $ \| \Psi \|_{C^2_b} $ 
the quantity $ E^{1/2} $ equivalent to the $ H^2 $-norm of $ R $ we are done
with the proof of  the Theorems \ref{main_theorem}-\ref{main_theoremwhitham}.
\qed

\section{Discussion}

It is the purpose of this section to give some heuristic arguments why the previous 
approach works and to put the approach in some larger framework.

The error equation \eqref{smitseq} to the spatially homogeneous Boussinesq equation \eqref{const}
can be written in lowest order in the form of a Hamiltonian system, namely 
\begin{eqnarray*}
\partial_t  \left( \begin{array}{c}  R \\ w \end{array} \right) =  \left( \begin{array}{c} 
w \\
 \partial_x^2 R - \partial_x^4 R + \varepsilon^{\alpha} \Psi \partial_x^2 R + \mathcal{O}(\varepsilon^{1+\alpha})
\end{array}
\right) = \left( \begin{array}{cc} 0 & 1 \\ -1 & 0  \end{array}  \right)\left( \begin{array}{c} \partial_R H \\ \partial_w H \end{array}  \right),
\end{eqnarray*}
with the Hamiltonian 
$$ 
H = \frac{1}{2} \int w^2 + (\partial_x R)^2+ (\partial_x^2 R)^2 + \varepsilon^{\alpha} \Psi (\partial_x R)^2 dx 
$$ 
where for this presentation we used $ \partial_x \Psi  =  \mathcal{O}(\varepsilon) $.
This Hamiltonian is a part of our energy and  it can be used to estimate parts of the $ H^2 $ norm.
Since $ \Psi  $ depends on $ t $ the Hamiltonian is not conserved, but 
we have  
\begin{equation}\label{fjs23}
\frac{d}{dt}  H = \nabla H \cdot \partial_t  \left( \begin{array}{c}  R \\ w \end{array} \right) + \partial_t H 
= 0 + \mathcal{O}(\varepsilon^{1+\alpha})
\end{equation}
since $ \partial_t \Psi  =  \mathcal{O}(\varepsilon) $  due to  the long wave character of the approximation.

In a similar way the spatially periodic case can be understood.
The error equation \eqref{error} to the spatially homogeneous Boussinesq equation \eqref{Boussinesq}
can be written in lowest order in the form of a Hamiltonian system, namely 
\begin{eqnarray*}
\partial_t  \left( \begin{array}{c}  R \\ w \end{array} \right) =  \left( \begin{array}{c} 
w \\
{\partial_x  (\mathcal{A}^2  (\partial_x R)) + \mathcal{O}(\varepsilon^{1+\alpha})}
\end{array}
\right) = \left( \begin{array}{cc} 0 & 1 \\ -1 & 0  \end{array}  \right)\left( \begin{array}{c} \partial_R H \\ \partial_w H \end{array}  \right)+\mathcal{O}(\eps^{1+\alpha}),
\end{eqnarray*}
with the Hamiltonian 
$$ 
H = \frac{1}{2} \int w^2 + (\mathcal{A} \partial_x R)^2 dx $$ 
where for this presentation we used $ \partial_x \Psi  =  \mathcal{O}(\varepsilon) $.
This Hamiltonian is a part of our energy and  it can be used to estimate parts of the $ H^2 $ norm.
Since $ \mathcal{A}$  depends via $ \Psi  $ on $ t $ the Hamiltonian is not conserved, but 
again we have  \eqref{fjs23}
since $ \partial_t \Psi  =  \mathcal{O}(\varepsilon) $  due to  the long wave character of the approximation.

As already said the paper was originally  intended as the next step in generalizing a method 
which has been developed in \cite{CS11} for the justification of the KdV approximation 
in situations when the KdV modes are resonant to other long wave modes respectively in \cite{DSS16}
for the justification of the Whitham  approximation. 
The normal form transforms which were used 
in  the proofs of \cite{CS11,DSS16} leave   the energy surfaces invariant  and can therefore be avoided by 
our 'good' choice of energy.
Hence also the toy problem considered in \cite{CS11,DSS16}  can be handled with the presented approach 
if the nonlinear terms are modified in such a way that a Hamiltonian structure is observed.

\appendix

\section{The inviscid Burgers approximation} \label{appA}

It is the goal of this appendix to provide more details about the 
derivation and the justification via error estimates for 
the inviscid Burgers approximation.
Inserting the ansatz 
\begin{equation*}
\varepsilon^{\alpha} \Psi(x,t) = \varepsilon^{\alpha}
A(\varepsilon(x-t),\varepsilon^{1+\alpha} t)
\end{equation*}
with $ \alpha \in (0,2) $ into  
the homogeneous Boussinesq equation 
\eqref{Boussinesq}
gives the residual 
\begin{eqnarray*}
\textrm{Res}(u)(x,t) & = & 
- \partial_t^2 u(x,t) +
\partial_x^2 u(x,t)
-\partial_x^4 u(x,t)
+\partial_x^2(u(x,t)^2) \\
& = & \varepsilon^{\alpha+ 4} \partial_X^4 A + \varepsilon^{3 \alpha +2 } \partial_T^2 A 
\end{eqnarray*}
and $ A $ to satisfy the inviscid Burgers equation 
$$
\partial_T A = 
-\frac{1}{2}\partial_X(A^2)
$$
if the coefficient of $ \varepsilon^{2 \alpha+ 2} $ is put to zero.
However, the residual is too large for the analysis made in Section \ref{sec2}.
By adding higher order terms to the approximation we obtain the estimates 
stated in Remark \ref{alphaville}, namely 
$$ \| \mathrm{Res}(  \varepsilon^{\alpha}   \Psi(\cdot,t,\varepsilon)   )     \|_{L^2} 
= \mathcal{O}( \varepsilon^{(7+4\alpha)/2}) \quad 
\textrm{and} 
\quad 
 \|  \partial_x^{-1}  \mathrm{Res}( \varepsilon^{\alpha}    \Psi(\cdot,t,\varepsilon)    )    \|_{L^2} = \mathcal{O}( \varepsilon^{(5+4\alpha)/2}) .
$$
We consider the improved approximation 
\begin{equation*}
\varepsilon^{\alpha} \Psi(x,t) = \varepsilon^{\alpha}
A(\varepsilon(x-t),\varepsilon^{1+\alpha} t) + \varepsilon^{\beta}
B(\varepsilon(x-t),\varepsilon^{1+\alpha} t)
\end{equation*}
with $ \beta = \min\{2 \alpha , 2  \} $. For the residual we find 
\begin{eqnarray*}
\mathrm{Res}(\varepsilon^2 \Psi) &  = & - 2 \varepsilon^{2 + \alpha + \beta} \partial_T \partial_X B 
- \varepsilon^{2 + 2 \alpha + \beta} \partial_T^2 B   - \varepsilon^{4 +\beta} \partial_X^{4} B 
+ 2 \varepsilon^{2 + \alpha + \beta } \partial_X^2 (A B) \\ &&  + \varepsilon^{2 + 2 \beta } \partial_X^2 (B^2) 
+ \varepsilon^{\alpha+ 4} \partial_X^4 A + \varepsilon^{3 \alpha +2 } \partial_T^2 A .
\end{eqnarray*}
We choose $ B $ to satisfy 
$$ 
2 \partial_T B = 2 \partial_X (A B)  + g 
$$ 
where 
$$ 
g = \left\{\begin{array}{cl}  \partial_X^{-1}\partial_T^2 A,& \textrm{for } \alpha \in (0,1), \\
 \partial_X^{-1}\partial_T^2 A + \partial_X^3 A ,& \textrm{for } \alpha = 1, \\
\partial_X^3 A ,& \textrm{for } \alpha \in (1,2).
\end{array} \right.
$$
By this choice we have 
$$ 
|\mathrm{Res}(\varepsilon^2 \Psi)| = \mathcal{O}(\max\{\chi_{\alpha \neq 1}(\alpha)\min\{\varepsilon^{\alpha+ 4},\varepsilon^{3 \alpha +2 } \},\varepsilon^{2 + 2 \alpha + \beta},\varepsilon^{4 +\beta} ,  \varepsilon^{2 + 2 \beta }\})  .
$$
Hence only for $ \alpha = 1 $, where $ \beta = 2 $, this is of order 
$ \mathcal{O}( \varepsilon^{4+ 2\alpha}) $
which is the formal order which is necessary to obtain the $ L^2 $ bound.
For all other values of $ \alpha $ more additional terms are necessary.
For $ \alpha \to 0 $ and $ \alpha \to 2 $ the number of such terms goes to infinity
and more and more regularity is necessary. 
We refrain from discussing the solvability of this system of amplitude equations.
This question is non-trivial since already for $ \alpha = 1 $ the term 
$ \partial_X^{-1}\partial_T^2 A $ has to be computed which is possible due to the fact that 
the temporal  derivatives can be expressed as spatial derivatives 
via the inviscid Burgers equation, namely 
$$  
\partial_T^2 A = -\frac{1}{2}\partial_T \partial_X(A^2) =  -\partial_X(A \partial_T  A) 
= \frac{1}{2}\partial_X(A \partial_X(A^2)) = \frac{1}{3}\partial_X^2 (A^3).
$$
Due to this presentation also the estimate for $  \partial_x^{-1}  \mathrm{Res}( \varepsilon^{\alpha}    \Psi ) $ can be obtained since now also 
$  \partial_T^2 B $ can be expressed as spatial derivatives.

\section{Higher regularity results}

It is the purpose of this section to explain how the approximation results can be transferred from 
$ H^2 $ to $ H^m $ with $ m \geq 2 $. Due to the $ x $-dependent coefficients energy estimates for
the spatial derivatives turn out to be rather complicated.
However, by considering  time derivatives the previous ideas  and energies still can be used.
The spatial derivatives then can be estimated via the equation for the error, namely
\begin{eqnarray} \label{doors1}
L R  & = & \partial^2_tR  -2\partial_x(c\partial_x(\varepsilon^{\alpha}\Psi R)) - R \\ && \qquad  -\varepsilon^{(3+2\alpha)/2}\partial_x(c\partial_x(R^2))-\varepsilon^{-(3+2\alpha)/2}\text{Res}(\varepsilon^2\Psi). \nonumber
 \end{eqnarray}
where 
$$ 
L R = \partial_x(a\partial_xR)-\partial_x^2(b\partial_x^2R) - R.
$$ 
The operator $ L $ is invertible and  maps $ H^s $ into $ H^{s+4} $, respectively $ C^m([0,T_0/\varepsilon^{1+\alpha}],H^s) $ into $ C^m([0,T_0/\varepsilon^{1+\alpha}],H^{s+4}) $.
For  $ R \in  C^m([0,T_0/\varepsilon^{1+\alpha}],H^s) $ the right-hand side of \eqref{doors1} is in 
$$
 C^{m-2}([0,T_0/\varepsilon^{1+\alpha}],H^{s}) \cap C^m([0,T_0/\varepsilon^{1+\alpha}],H^{s-2}).
$$
An application of $ L^{-1} $ to \eqref{doors1} shows that 
$$ 
R  \in C^{m-2}([0,T_0/\varepsilon^{1+\alpha}],H^{s+4}) \cap C^m([0,T_0/\varepsilon^{1+\alpha}],H^{s+2}).
$$ 
Iterating this process shows that temporal derivatives can be transformed into spatial derivatives.

It remains to obtain the estimates for the temporal derivatives. In order to do so 
we differentiate 
the equation for the error $ m $ times w.r.t. $ t $. We obtain an equation of the form  
\begin{eqnarray}\label{errort}
\partial^2_t (\partial^m_t R)  & = &\partial_x (\mathcal{A}^2(\partial_x( \partial^m_t R)))+2\partial_x(c(\partial_x\varepsilon^{\alpha}\Psi) ( \partial^m_t R)) + \mathcal{O}(\varepsilon^{1+\alpha})
\end{eqnarray}
due to the fact that  whenever a time derivative falls on $ \mathcal{A} $ or $ \Psi $ another $ \varepsilon $ is gained.
In order to bound the solutions of \eqref{errort} we use energy estimates.
Therefore, we first multiply \eqref{errort}
 with $\partial_t^{m+1}R$ and  integrate the obtained expression w.r.t. $x$. Next as above 
 we multiply ''$\partial_x^{-1} $\eqref{errort}''
 with $\mathcal{A}^{-2} \partial_x^{-1}\partial_tR$ and  integrate the expression obtained  in this way w.r.t. $x$.
 
 If we define 
\begin{align*}
 E_m(t)\,\,=\,\, \frac 1 2 \left(||\partial_t^{m+1}R||_{L^2}^2\,+||\mathcal{A}^{-1}\partial_x^{-1}\partial_t^{m+1}R||_{L^2}^2\,+\,||\partial_t^{m}R||_{L^2}^2\,+\,||\mathcal{A}\partial_x \partial_t^{m}R||_{L^2}^2\right).
\end{align*}
we find 
\begin{eqnarray*}
\frac{d}{dt}{E}_m & \leq & C_1 \varepsilon^{1+\alpha} E_m + C_2 \varepsilon^{(3+2\alpha)/2} \mathcal{E}_m^{3/2} +  C_3 \varepsilon^{1+\alpha} ,
\end{eqnarray*}
with constants $ C_1 $, $ C_2 $, and $ C_3 $ independent of $ 0 < \varepsilon \ll 1 $
and  $\mathcal{E}_m = E + \ldots + E_m $. 
Summing up all estimates for the $ E_j $ for $ j = 0,\ldots, m $ yields a similar inequality for  $ \mathcal{E}_m $
Applying Gronwall's inequality to this inequality gives for instance 
\begin{theorem}\label{main_theoremtt}
Fix $ s \in \mathbb{N} $ and let  $A\in C([0,T_0],H^{6+s}(\mathbb R)) $  
be a solution of the  KdV equation  \eqref{main_KdV}.
Then there exist $\varepsilon_0>0$, $C>0$ such that for all $\varepsilon \in (0,\varepsilon_0)$ we have  
solutions $ u \in C([0,T_0/\varepsilon^3],H^{2+s})$ of the spatially periodic Boussinesq model \eqref{Boussinesq}
with
\begin{align*}
\sup_{t\in [0,T_0/\varepsilon^3]}\|u(\cdot,t)-\varepsilon^2 A(\varepsilon(\cdot-t),\varepsilon^3t)\tilde f_1(0)(\cdot)\|_{H^{2+s}}\leq C\varepsilon^{5/2}.
\end{align*}
\end{theorem}
Theorem \ref{main_theoremburgers} and Theorem \ref{main_theoremwhitham} can be reformulated in a similar way.

\section{Bloch transform on the real line}

\label{leutkirch}

In this section we recall basic properties of Bloch transform. 
Our presentation follows \cite{GSU}.
Bloch transform $ \mathcal{T} $ generalizes  Fourier transform $ \mathcal{F} $ from spatially
homogeneous problems to  spatially periodic problems.
Bloch transform is (formally) defined by
\begin{equation} \label{bt1}
\widetilde{u}(\ell,x)=(\mathcal{T} u)(\ell,x)=\sum_{j\in\Z}e^{i j x}\widehat{u}(\ell+j),
\end{equation}
where $\widehat{u}(\xi) = \left( \mathcal{F} u \right)(\xi)$, $\xi \in \mathbb{R}$ is the Fourier transform
of $u$. The inverse of Bloch transform is given by
\begin{equation}
u(x)=(\mathcal{T}^{-1}\widetilde{u})(x)=\int_{-1/2}^{1/2}e^{i \ell x}\widetilde{u}(\ell,x)
d \ell.
\end{equation}
By construction, $\widetilde{u}(\ell,x)$ is extended from $(\ell,x) \in \mathbb{T}_1 \times \mathbb{T}_{2\pi}$
to $(\ell,x) \in \mathbb{R} \times \mathbb{R}$ according to the continuation conditions:
\begin{gather}
\widetilde{u}(\ell,x) = \widetilde{u}(\ell,x + 2 \pi)   \quad
\mbox{and} \quad
\widetilde{u}(\ell,x) = \widetilde{u}(\ell + 1,x) e^{i x}.
\label{bt2si}
\end{gather}
The following lemma specifies the well-known property of Bloch transform acting on Sobolev function spaces.
\begin{lemma} \label{iso}
Bloch transform $\mathcal{T}$ is an isomorphism between
$$H^s(\R) \qquad  \textrm{and}  \qquad L^2(\mathbb{T}_1,H^s(\mathbb{T}_{2 \pi})),
$$
where $L^2(\mathbb{T}_1,H^s(\mathbb{T}_{2 \pi}))$ is equipped with the norm
$$
\| \widetilde{u} \|_{ L^2(\mathbb{T}_1,H^s(\mathbb{T}_{2 \pi}))}=
\left(\int_{-1/2}^{1/2}
 \|\widetilde{u}(\ell,\cdot )\|_{H^s(\mathbb{T}_{2 \pi})}^2 d\ell\right)^{1/2}.
 $$
\end{lemma}
Multiplication of two functions $u(x)$ and $v(x)$ in $x$-space corresponds
some convolution  in Bloch space:
\begin{gather}
(\tilde{u} \star \tilde{v}) (\ell, x) =
\int\limits_{- 1/2}^{1/2} \tilde{u} (\ell - m,x)
\tilde{v}(m,x) d m,
\label{bconv1}
\end{gather}
where the continuation conditions \eqref{bt2si} have to be used for $|\ell-m|>1/2$. If
$\chi:\R\rightarrow\R$ is $2\pi$ periodic, then
\begin{equation} \label{tweedy}
\mathcal{T}(\chi u)(\ell,x)=\chi(x)(\mathcal{T} u)(\ell,x).
\end{equation}
The relations (\ref{bconv1}) and (\ref{tweedy}) are well-known  and
can be proved from the definition (\ref{bt1}).

\end{document}